\documentclass[10pt]{article}
\usepackage{amsthm,amsmath,amsfonts,epstopdf,verbatim,enumitem,etoolbox}

\usepackage[cp1250]{inputenc}
\usepackage[IL2]{fontenc}
\usepackage[english]{babel}

\allowdisplaybreaks[4]    

\usepackage[colorlinks=true]{hyperref}
\hypersetup{urlcolor=blue, citecolor=red}

\newtheoremstyle{newstyle} 
    {\topsep}                    
    {\topsep}                    
    {\slshape}                   
    {}                           
    {\bfseries}                   
    {.}                          
    {5pt plus 1pt minus 1pt}                       
    {}  

\theoremstyle{newstyle}
\newtheorem{tvrz}{Proposition}[section]
\newtheorem{veta}[tvrz]{Theorem}

\newtheorem{lemmz}[tvrz]{Lemma}
\newtheorem{cor}[tvrz]{Corollary}
\theoremstyle{definition}
\newtheorem{defi}[tvrz]{Definition}
\newtheorem{pozn}[tvrz]{Remark}
\newtheorem{priklad}[tvrz]{Example}
\newtheorem*{defi*}{Definition}

\def\U{ \mathcal{U}}
\def\A{ \mathcal{A}}
\def\K{ \mathcal{K}}

\def\B{ \mathcal{B}}

\def\O{ \mathcal{O}}
\def\N{ \mathbb{N}}
\def\R{ \mathbb{R}}
\def\Rn{ \mathbb{R}^n}
\def\ddd{ \,\mathrm{d}}
\def\cara{\, ; \;}
\def\D{\mathrm{D}}

\DeclareMathSymbol{\varnothing}     {\mathord}{AMSb}{"3F}
\DeclareMathSymbol{\varkappa}       {\mathord}{AMSb}{"7B}
\DeclareMathSymbol{\smallsetminus}  {\mathbin}{AMSb}{"72}

\begin{document}


\title{Generalized Ordinary Differential Equations in Metric Spaces}
\author{\textbf{Břetislav Skovajsa}\thanks{The author was supported by the ERC CZ grant LL1203 of the Czech Ministry of Education.} \\ 
Charles University, Faculty of Mathematics and Physics \\
Department of Mathematical Analysis\\
Sokolovská 83, 186 00 Prague 8, Czech Republic\\
E-mail: skovajsa@karlin.mff.cuni.cz}

\date{}

\maketitle

\begin{abstract}
The aim of this text is to extend the theory of generalized ordinary differential equations to the setting of metric spaces. We present existence and uniqueness theorems that significantly improve previous results even when restricted back to Euclidean spaces. \\[6pt]
\end{abstract}

\section*{Introduction}

The theory of generalized ordinary differential equations was established in 1957 by J.~Kurz\-weil \cite{kurzweil57}. It is one of the most comprehensive overarching theories in terms of admissible data and structure of the equation, which makes it particularly useful for dealing with phenomena that lead to low regularity or even discontinuity of solutions such as impulses or high oscillation. However, this theory was so far only considered on linear spaces and expanding it to metric spaces will considerably increase the breadth of problems available for study in this setting.  \par
The reason behind the versatility of Kurzweil's theory can be traced back to its two main benefits.
\begin{enumerate}[label=(\roman*), itemsep=2pt]
\item
First, the problem is formulated in a very general manner. A generalized ordinary differential equation (GODE for short) is given by a function $F \colon \Rn \times [a,b] \times [a,b] \rightarrow \Rn$ and solving it can be roughly understood as looking for a function $u \colon [a,b] \rightarrow \Rn$ satisfying
\begin{equation} \label{godesim1}
u(t) \; \sim \; u(\tau) + F(u(\tau),\tau, t) - F(u(\tau), \tau , \tau) \quad \text{for} \quad t \to \tau.
\end{equation}
In geometrical terms, every point of $\Rn \times [a,b]$ is associated with different and possibly nonlinear infinitesimal behaviour, given by a ``tangent curve" i.e.\ the right hand side of \eqref{godesim1} as a function of $t$. In this context, the standard ordinary differential equation  
\begin{equation} \label{careq}
\dot{x} = f(x,t)
\end{equation}
corresponds to the special case $F(x, \tau, t) = f(x, \tau) \, t$.
\item
Secondly, we are dealing with a more general concept of solution, as the relation between the solution and the right hand side is understood in terms of the Henstock-Kurzweil integral. This puts very little qualitative restraint on the data of the equation. Furthermore, it allows this concept to encompass both classical solutions and solutions in terms of measure due to the ability of nonabsolutely convergent integrals to both include the Lebesgue integral and integrate all derivatives.
\end{enumerate}
While early contributions by J.\ Kurzweil, J.\ Jarník and I.\ Vrkoč mainly focused on applications to continuous dependence of ODEs on a parameter \cite{jarnik1,jarnik2,jarnik3,vrkoc}, the methods of GODEs were soon after extended to Banach spaces, where they could be applied to certain types of PDEs as well \cite{kurzpde1,kurzpde2,kurzpde3}. For a comprehensive summary of subsequent development, we point to key monographs \cite{schwabik92} and \cite{STV} by Š.\ Schwabik, M.\ Tvrdý and O.\ Vejvoda. The theory of GODEs remains relevant to this day, as evidenced by more recent books \cite{tvrdybook,kurzweil11} and numerous articles by M.\ Tvrdý, A.\ Slavík, G.\ A.\ Monteiro, M.\ Federson and many others. Currently, it is known that GODEs include an extremely wide range of problems, such as differential equations with impulses \cite{slavik1}, dynamic equations on time scales \cite{slavik2}, Fredholm-Stieltjes and Volterra-Stieltjes equations \cite{tvrdybook,federson1,federson2} and many types of functional differential equations \cite{vorel1,vorel2,federson,slavik4}. There are even cases where a single result concerning GODEs managed to encompass several theorems for seemingly unrelated types of equations \cite{slavik5}. \par
Recently, a growing amount of attention is called to analysis on metric spaces and many theories clasically associated with linear structure are being transferred to this more general setting \cite{MS1,MS2,MS3,MS4,MS5,MS6}. This has the effect that various special structures (like manifolds or spaces with sub-Riemannian geometry) are handled all at once. Theories motivated by ODEs in particular were developed e.g.\ by A.\ I.\ Panasyuk \cite{panasyuk1,panasyuk2,panasyuk3}, J.\ P.\ Aubin \cite{aubin92,aubin93,aubin99}, T.\ Lorenz \cite{lorentz1,lorentz2} and J.\ Tabor \cite{tabor01}. While terminology differs widely and ranges from differential inclusions to constructions resembling tangent spaces of manifolds, they all share the same common idea. In the absence of lines, they also resort to arbitrary tangent curves. However, this involves an additional layer of technical difficulty. Any curve that approximates the behaviour of the solution around a specific point needs to pass through that point. This is always true in \eqref{godesim1}, as the tangent curve is generated by addition. However, in metric spaces, this needs to be explicitly required. As such, they are concerned with the tangent behaviour 
\begin{equation} \label{metsim}
u(\tau + s) \; \sim \; Q(u(\tau),\tau, s) \quad \text{for} \quad s \to 0,
\end{equation}
where $(X, \mathbf{q})$ is a metric space, $u \colon \R \rightarrow X$ and the mapping $Q \colon X \times \R \times [- 1, 1] \rightarrow X$ satisfies the condition
\begin{equation} \label{normalizemp}
Q(x, \tau, 0) = x.
\end{equation}
With the help of \eqref{normalizemp}, the relation in \eqref{metsim} can be defined by the pointwise limit
\begin{equation} \label{metriclim}
\lim_{s \to 0} \frac{ \mathbf{q} \Big( u(\tau + s), Q(u(\tau), \tau, s)  \Big) }{|s|} = 0.
\end{equation}
By requiring \eqref{metriclim} to hold everywhere or almost everywhere it is possible to obtain solutions in the classical or Carath\'{e}odory sense respectively. We can see that the concept of tangent curves appears in both respective theories with quite different motivation, which Panasyuk already noticed \cite{panasyuk1}. However, in terms of quality of solution, none of the previous studies on metric spaces reached the generality of Kurzweil's theory. \par
Our goal is to merge these two generalizations of differential equations into a single theory. First, we notice that every GODE on a linear space can be easily transformed so that it satisfies $F(x, \tau, \tau) = x$. The problem \eqref{godesim1} then transforms into
\begin{equation*} 
u(t) \; \sim \; F(u(\tau),\tau, t) \quad \text{for} \quad t \to \tau,
\end{equation*}
which can be posed in a metric space. We then demonstrate that despite the inability to integrate over the target space, we can express this problem as an integral equation in one dimension involving the distance of correctly chosen elements. We even obtain a pointwise representation similar to \eqref{metriclim} with the help of the MC integral by J.~Mal\'{y} and H.~Bendov\'{a} \cite{malybendova}. Contrary to expectation, this results in a surprisingly elegant formulation that avoids unnecessary technicalities. \par
As our main achievement, we present uniqueness and existence theorems which not only replicate known results of the GODE theory in a new setting, but also offer considerable improvements, even when restricted to Euclidean spaces. The most notable improvement concerns the quality of solutions. Standard conditions for uniqueness and existence require the tangent curves to be functions of bounded variation, which directly causes the solutions to be BV functions as well. In contrast, we work with relaxed conditions that only require the tangent curves to be regulated, which presents another added layer of difficulty. The qualitative jump from BV functions to regulated functions remains a long standing goal in many areas of GODEs and Henstock-Kurzweil integration. As such, our results are of significant interest even without the context of metric spaces.  \par
For a more practically oriented motivation of dealing with regulated functions, we can mention the theory of measure differential equations \cite{MDE1}. It was developed as a method for dealing with impulsive perturbations of \eqref{careq} and considers the problem 
\begin{equation} \label{mdeq}
\D x = f(x, t) + h(x,t) \, \D g,
\end{equation}
where $\D x$ and $\D g$ are understood as distributional derivatives of the functions $x$ and $g$ respectively. In order to express \eqref{mdeq} with the help of the Lebesgue-Stieltjes integral, $g$ is required to be a function of bounded variation \cite{MDE2}. However, the Stieltjes version of the Henstock-Kurzweil integral does not share this limitation. Thus, the GODE given by 
\begin{equation*}
F(x, \tau, t) = f(x, \tau)\, t + h(x, \tau) \, g(t),
\end{equation*} 
allows us to examine a much wider range of perturbations.\par
The first section deals with preliminaries. The second section recalls the basics of the MC integral and slightly modifies it to make it more compatible with the GODE theory. The third section recalls the basics of generalized ordinary differential equations and presents the standard theorems that we will aim to improve. The fourth section deals with the technicalities of expanding the definition of GODEs to metric spaces. The fifth and sixth chapter deal with uniqueness and existence of solutions respectively and contain our two main results: Theorem \ref{mainunique} and Theorem \ref{mainexist}. The last chapter shows how our theorems look when restricted to linear spaces and how they compare to previous results.

\section{Preliminaries}
For $x \in \R$ we say that $x$ is positive if $x > 0$ and nonnegative if $x \geq 0$. The symbol $\R^+$ then stands for the set of positive real numbers and $\R^+_0$ for the set of nonnegative real numbers. By $\U(x, r)$ and $\B(x, r)$ we denote the open and closed ball with centre $x$ and radius $r$. A function $f \colon \R \rightarrow \R$ is increasing if $f(s) < f(t)$ for $s < t$ and nondecreasing if $f(s) \leq f(t)$ for $s < t$. For a real function $h \colon \R \rightarrow \R$ we will use $h(x+)$ to denote the limit of $h$ at the point $x \in \R$ from the right, if it exists. By $\D^+ f(\tau)$ we denote the upper right derivative of the function $f$ at the point $\tau \in \R$ i.e.\
\begin{equation*}
\D^+ f(\tau) = \limsup_{t \to \tau_+} \frac{f(t) - f(\tau)}{t - \tau}.
\end{equation*}
We say that the function $f \colon \Rn \times [a,b] \rightarrow \Rn$ satisfies the standard Carath\'{e}odory assumptions if 
\begin{enumerate}[label=$\mathbf{(C \arabic*)}$, noitemsep]
\item the function $x \mapsto f(x, t)$ is continuous for almost all $t \in [a,b]$,
\item the function $t \mapsto f(x, t)$ is measurable for all $x \in \Rn$,
\item there exists a Lebesgue integrable function $m \colon [a,b] \rightarrow \R^+_0$ such that 
\begin{equation*}
\|f(x, t)\| \leq m(t) \quad \text{for all} \;  x \in \Rn \; \text{and almost all} \; t \in [a,b].
\end{equation*}
\end{enumerate}
We recall that under these assumptions for every $x_0 \in \Rn$ and every $t_0 \in [a,b]$ there exists $\Delta > 0$ and a function $x \colon [a,b] \cap [t_0 - \Delta, t_0 + \Delta] \rightarrow \Rn$ such that
\begin{equation*}
x(t) = x_0 + \mathrm{(L)} \int^t_{t_0} f(x(s), s) \ddd s \quad \text{for} \quad  t \in  [a,b] \cap [t_0 - \Delta, t_0 + \Delta].
\end{equation*} 
The term partition of $[a,b] \subset \R$ will stand for any collection of closed intervals and tags $\{[t_{i-1}, t_i], \tau_i\}_{i=1}^k$ such that $t_0 = a$, $t_k = b$ and $\tau_i \in [t_{i-1}, t_i]$. 

\begin{defi}
Let $\delta \colon [a,b] \rightarrow \R^+$ be a positive real function defined on $[a,b]$. A partition of $[a,b]$ is called $\delta$-fine if for each $i = 1, \ldots , k$ it satisfies 
\begin{equation} \label{deltafine}
[t_{i-1}, t_i] \subset (\tau_i - \delta(\tau_i), \tau_i + \delta(\tau_i)).
\end{equation}
\end{defi}

\begin{lemmz}[Cousin] \label{cousin}
For every $\delta \colon [a,b] \rightarrow \R^+$ the set of all $\delta$-fine partitions of $[a,b]$ is nonempty.
\end{lemmz}

For proof see \cite{kurzweil57} (Lemma 1.1.1) or \cite{henstocklectures} (Theorem 3.1). \\[6pt]
In the entire sequel, $E$ will denote a normed linear space and $I \subset \R$ will denote an arbitrary interval (we do not limit ourselves to closed intervals as domains for the indefinite integral).

\begin{defi}[J. Kurzweil]

A function $u \colon I \rightarrow E$ is called an indefinite SHK integral of $U \colon I \times I \rightarrow E$ if for every $\varepsilon > 0$ and every $[a , b ] \subset I$ there exists $\delta \colon [a , b ] \rightarrow \R^+$ such that for every $\delta$-fine partition $\{[t_{i-1}, t_i], \tau_i\}_{i=1}^k$ of $[a , b ]$ we have
\begin{equation} \label{skhdef}
\sum_{i = 1}^k \|u(t_i) - u(t_{i-1}) - U(\tau_i, t_i) + U(\tau_i, t_{i-1}) \| < \varepsilon. 
\end{equation}
The definite SHK integral of $U$ over $[\alpha , \beta ] \subset I$ is defined as 
\begin{equation*}
(\text{SHK})\int_{ \alpha }^{ \beta } \D_t \, U(\tau, t) = u( \beta ) - u( \alpha ).
\end{equation*}  
\end{defi}

\begin{pozn}
$\mathbf{a)}$ The function $\delta$ is usually referred to as \textit{gauge} and the resulting construction can also be found in literature under the name \textit{gauge integral}. The same concept was independently, and for unrelated reasons, discovered by R. Henstock \cite{henstock61}. This version of the gauge integral is not the most general possible, as we could instead consider integrands that depend on point-interval pairs, but it is the one most suited to the theory of differential equations. \\[6pt]
$\mathbf{b)}$ For $U(\tau,t) = f(\tau) \, t$ this might seem like a simple modification of the Riemann definition, but the resulting integral is equivalent to the Perron integral. \\[6pt]
$\mathbf{c)}$ This definition not in conflict with the standard definition in which only partitions of the whole interval $I = [a,b]$ are considered, since any $\delta$-fine partition of $[\alpha, \beta] \subset [a,b]$ can be extended into a $\delta$-fine partition of $[a,b]$ with the help of Lemma \ref{cousin}. \\[6pt]
$\mathbf{d)}$ On linear spaces, many works consider a slightly wider definition of the Henstock-Kurzweil integral that directly establishes the definite integral $A \in E$ by demanding
\begin{equation*}
\Bigl\| A - \sum_{i = 1}^k  U(\tau_i, t_i) - U(\tau_i, t_{i-1}) \Bigr\| < \varepsilon. 
\end{equation*}
However, the SHK (Strong Henstock-Kurzweil) integration is better suited for problems that deal with abstract valued functions. For $E = \Rn$ these methods of integration are equivalent as a trivial consequence of the Saks-Henstock Lemma. 
\end{pozn}

\section{MC Integral}

The monotonically controlled (MC for short) integral was introduced in \cite{malybendova} by J.~Mal\'{y} and H.~Bendov\'{a}. Their aim was to build the foundations of integral theory at the generality of Perron integral while using unexpectedly simple definitions and proofs. They prove that the MC integral is equivalent to the SHK integral in the Stieltjes case. In this section we generalize the definition of the MC integral and the equivalence result to the case of coupled variables to make them compatible with the GODE theory.

\begin{defi}
A function $u \colon I \rightarrow E$ is called an indefinite MC integral of $U \colon I \times I \rightarrow E$ if there exists an increasing function $\xi \colon I \rightarrow \R$, called control function of $(U, u)$ on $I$, such that
\begin{equation*}
\lim_{t \, \to \, \tau, \; t \, \in \, I} \frac{ \|u(t) - u(\tau) - U(\tau, t) + U(\tau, \tau) \|  }{\xi(t) - \xi(\tau)} = 0 \quad \text{for} \quad \tau \in I. 
\end{equation*}
We define the definite MC integral of $U$ over $[\alpha , \beta ] \subset I$ as 
\begin{equation*}
(\text{MC})\int_{ \alpha }^{ \beta } \D_t \, U(\tau, t) = u( \beta ) - u( \alpha ).
\end{equation*}  
Note that if $\xi$ is a control function of $(U, u)$ on $I$, $\alpha > 0$ and $\zeta \colon I \rightarrow \R$ is a nondecreasing function, then $\alpha \xi + \zeta$ is also a control function of $(U,u)$ on $I$.

\end{defi}

\begin{veta} \label{integralequiv}
A function $u \colon I \rightarrow E$ is an indefinite MC integral of $U \colon I \times I \rightarrow E$ on $I$ if and only if it is an indefinite SHK integral of $U$ on $I$.
\end{veta}

\begin{proof}
First, we assume that $u$ is an indefinite MC integral of $U$ on $I$ with a control function $\xi$ satisfying $0 \leq \xi \leq 1$.\par
Choose $\varepsilon > 0$. Then for each $\tau \in I$ there exists $\delta(\tau) > 0$ such that every $t \in (\tau - \delta(\tau), \tau + \delta(\tau)) \cap I$ satisfies
\begin{equation*}
\|u(t) - u(\tau) - U(\tau, t) + U(\tau, \tau) \| < \varepsilon |\xi(t) - \xi(\tau)|.
\end{equation*}    
For $[\alpha , \beta] \subset I$ and a $\delta$-fine partition $\{[t_{i-1}, t_i], \tau_i\}_{i = 1}^k$ of $[\alpha, \beta]$ we can see that
\begin{align*}
\|u(t_i) - u(\tau_i) - U(\tau_i, t_i) + U(\tau_i, \tau_i) \| &< \varepsilon (\xi(t_i) - \xi(\tau_i)), \\
\|u(\tau_i) - u(t_{i- 1}) + U(\tau_i, t_{i- 1}) - U(\tau_i, \tau_i) \| &< \varepsilon (\xi(\tau_i) - \xi(t_{i-1})).
\end{align*}
Therefore, we obtain
\begin{gather*}
\sum_{i = 1}^k \|u(t_i) - u(t_{i- 1}) - U(\tau_i, t_i) + U(\tau_i, t_{i-1}) \|  \\
< \varepsilon \sum_{i = 1}^k  (\xi(t_i) - \xi(\tau_i) + \xi(\tau_i) - \xi(t_{i-1})) = \varepsilon (\xi(b) - \xi(a)) \leq \varepsilon .
\end{gather*}
Now, let $u$ be an indefinite SHK integral of $U$ on $I$. Given $[\alpha, \beta] \subset I$ and a partition $\A = \{[t_{i-1}, t_i], \tau_i\}_{i=1}^k$ of $[\alpha, \beta]$, set
\begin{equation*}
\sum_\A (u, U) :=  \sum_{i=1}^k \|u(t_i) - u(t_{i- 1}) - U(\tau_i, t_i) + U(\tau_i, t_{i-1}) \|.
\end{equation*}
Denote $a = \inf I$ and $b = \sup I$. Let $\{a_n\}_{q = 1}^{\infty} \subset I$ be such that $a_n = a$ for every $n \in \N$ if $a \in I$ and $a_n \to a$ for $n \to \infty$, $a_{n + 1} < a_n$ if $a \notin I$. Similarly, let $\{b_n\}_{n = 1}^{\infty} \subset I$ be such that $b_n = b$ for every $n \in \N$ if $b \in I$ and $b_n \to b$ for $n \to \infty$, $b_{n + 1} > b_n$ if $b \notin I$. For every $n \in \N$ let $\delta_n \colon [a_n , b_n] \rightarrow \R^+$ correspond to $\varepsilon_k = 2^{-n}$. For $\tau \in (a_n , b_n]$ we define $\A_n(\tau)$ as the set of all $\delta_n$-fine partitions of $[a_n, \tau]$ and
\begin{equation*}
  \xi_n(x) = \left\{
  \begin{array}{l l}
    0, & \quad x \leq a_n,\\
    \displaystyle \xi_n(\tau) = \sup \, \Big\{\sum_\A (u, U) \cara \A \in \A_n(\tau) \Big\}, & \quad a_n < \tau \leq b_n , \\[-4pt]
    \xi_n(b_n), & \quad x > b_k. \\
  \end{array} \right.
\end{equation*} 
Finally, set
\begin{equation*}
\xi(\tau) = \tau + \sum_{n = 1}^{\infty} n \; \xi_n(\tau). 
\end{equation*}
Choose $\varepsilon > 0$ and $\tau \in I \smallsetminus \{a\}$. Find $n_0 \in \N$ and $\Delta > 0$ such that $(\tau - \Delta, \tau] \subset [a_{n_0}, b_{n_0}]$ and $1 / n_0 < \varepsilon$. Set $\delta = \min \, \{\delta_{n_0}(\tau), \Delta\}$. It follows that for $t \in (\tau - \delta, \tau)$ we have
\begin{gather*}
\|u(t) - u(\tau) - U(\tau, t) + U(\tau, \tau) \| \leq \xi_{n_0}(\tau) - \xi_{n_0}(t)  \\
\leq \sum_{n = 1}^{\infty} \frac{n}{n_0} (\xi_n(\tau)-\xi_n(t)) \leq \frac{1}{n_0}(\xi(\tau)-\xi(t))<\varepsilon (\xi(\tau)-\xi(t)).
\end{gather*}
Similarly, for $\tau \in I \smallsetminus \{b\}$ we find $n_0 \in \N$ and $\delta > 0$ such that 
\begin{equation*}
\|u(t) - u(\tau) - U(\tau, t) + U(\tau, \tau) \| \leq \xi_{n_0}(t) - \xi_{n_0}(\tau)  < \varepsilon (\xi(t) - \xi(\tau))
\end{equation*}
whenever $t \in (\tau , \tau + \delta)$.
\end{proof}

\begin{pozn}
From the proof of the previous theorem we can deduce that we could equivalently define the indefinite SHK integral on an arbitrary interval by demanding a single gauge that works on every closed subinterval. Proving this without the help of the MC integral is a standard but lengthy exercise associated with the Hake theorem.
\end{pozn}

\section{Generalized Ordinary Differential Equations} \label{secgode}

In this section we recall the main concepts of the theory of GODEs as developed in \cite{kurzweil57} and formulate the known
existence and uniqueness theorems which we want to generalize. Once again, we slightly modify the definitions to allow for intervals of arbitrary type. Roughly speaking, the GODE 
\begin{equation}\label{feq}
\dot{x} = \D_t \, F(x, \tau, t)
\end{equation}
is the task of finding a function $u$ satisfying \eqref{godesim1}. However, a great deal of the research is concerned with the \textit{restricted} GODEs
\begin{equation}\label{geq}
\dot{x} = \D_t \, G(x, t).
\end{equation}
The reason for this is that a large class of GODEs (see \cite{kurzweil11} chapter 23) can be simplified by putting
\begin{equation*}
G(x, t) = \mathrm{(SHK)} \int_a^t \D_s \, F(x, \tau, s).
\end{equation*} 
We are going to give the precise definitions now.
\begin{defi} \label{standardgode} Let $E$ be a normed linear space and $\Omega \subset E$. Let $F \colon \Omega \times I \times I \rightarrow E$ be given. We say that a function $u \colon I \rightarrow E$ is a solution of the equation \eqref{feq}
 on $I$ if $u(t) \in \Omega$ for all $t \in I$ and
\begin{equation*} 
u( \beta ) = u( \alpha ) + \text{(SHK)} \int_{ \alpha }^{ \beta } \D_t \, F(u(\tau), \tau, t) \quad \text{for} \quad [\alpha , \beta ] \subset I. 
\end{equation*}
Specifically, for $I = [a,b]$ it means that for every $\varepsilon > 0$ there exists $\delta \colon [a,b] \rightarrow \R^+$ such that every $\delta$-fine partition $\{[t_{i-1}, t_i], \tau_i\}_{i = 1}^k$ of $[a,b]$ satisfies 
\begin{equation} \label{sumcalcel}
\sum_{i=1}^k \| u(t_i) - u(t_{i-1}) - F(u(\tau_i), \tau_i, t_i) + F(u(\tau_i), \tau_i, t_{i-1} )  \| < \varepsilon.
\end{equation}
Let $G \colon \Omega \times I \rightarrow E$ be given. We say that a function $u \colon I \rightarrow E$ is a solution of the equation \eqref{geq} on $I$ if $u(t) \in \Omega$ for all $t \in I$ and
\begin{equation*}
u( \beta ) = u( \alpha ) + \text{(SHK)} \int_{ \alpha }^{ \beta } \D_t \, G(u(\tau), t) \quad \text{for} \quad [ \alpha, \beta ] \subset I. 
\end{equation*} 
We say that the equation \eqref{feq} is \textit{normalized} if $F(x,\tau,\tau) = x$ for each $(x,\tau) \in \Omega \times I$. We 
can normalize any equation of this type by setting
\begin{equation*}
\tilde{F}(x, \tau, t) = x + F(x, \tau, t) - F(x, \tau, \tau).
\end{equation*}
Due to cancellation in \eqref{sumcalcel}, the equations
\begin{equation*}
\dot{x} = \D_t \, F(x, \tau, t) \quad \text{and} \quad \dot{x} = \D_t \, \tilde F(x, \tau, t)
\end{equation*} have the same solutions. However, normalization of the restricted equation \eqref{geq} leads to the unrestricted equation
\begin{equation*}
\dot{x} = \D_t \, [x+G(x, t)-G(x,\tau)].
\end{equation*}
\end{defi}
Now, we are ready to formulate the existence and uniqueness results
that we want to extend.

\begin{defi}
We say that $\omega \colon \R^+_0 \rightarrow \R^+_0$ is a modulus function if it is continuous, nondecreasing, $\omega(0) = 0$ and $\omega(\nu) > 0$ for $\nu > 0$. \par
We say that $\omega \colon \R^+_0 \rightarrow \R^+_0$ is an Osgood type modulus function if it is a modulus function and for every $\nu > 0$ it satisfies
\begin{equation} \label{osgoodtype}
\lim_{r \to 0_+} \int^\nu_r \frac{1}{\omega(s)} \ddd s = + \infty. 
\end{equation}
\end{defi}

\begin{defi}
We say that $G \colon \O \rightarrow \Rn$ belongs to the class $\mathcal{F}(\O, h, \omega)$ if
\begin{gather} 
\| G(x, t) - G(x, s) \| \leq |h(t) - h(s)|,    \label{calf1} \tag{$\mathcal{F}_1$} \\
\| G(x, t) - G(x, s) - G(y, t) + G(y, s) \| \leq \omega(\|x - y \|) \, |h(t) - h(s)| \label{calf2} \tag{$\mathcal{F}_2$}
\end{gather}
for $(x, t), (x, s), (y, t), (y, s) \in \O$, where $\O = \U_r \times (a,b)$, $h \colon [a,b] \rightarrow \R$ is nondecreasing and $\omega$ is a modulus function. The symbol $\U_r$ stands for $\U(0, r) \subset \Rn$.
\end{defi}

\begin{veta} [Š.~Schwabik, \cite{schwabik92}, page 122] \label{schwabunique}
Let $G \colon \O \rightarrow \Rn$ belong to the class $\mathcal{F}(\O, h, \omega)$, where $\omega$ is an Osgood type modulus function and the function $h \colon [a,b] \rightarrow \R$ is nondecreasing and continuous from the left. Then every solution $x \colon (a,b) \rightarrow \Rn$ of $\dot{x} = \D_t \, G(x, t)$ such that $(x(\tau), \tau) \in \O$ and
\begin{equation} \label{initialjump}
x(\tau) + \lim_{t \to \tau_+} G(x(\tau), t) - G(x(\tau), \tau) \in \U_r
\end{equation}
is locally unique in the future at $(x(\tau), \tau)$. (In terms of Definition \ref{unique-future})
\end{veta}

\begin{veta} [Š.~Schwabik, \cite{schwabik92}, page 114] \label{schwabexist}
Let $G \colon \O \rightarrow \Rn$ belong to the class $\mathcal{F}(\O, h, \omega)$ and let $(x_{\tau}, \tau) \in \O$ be such that
\begin{equation*}
x_{\tau} + \lim_{t \to \tau_+} G(x_{\tau}, t) - G(x_{\tau}, \tau) \in \U_r.
\end{equation*}
Then there exists $\Delta > 0$ and a function $x \colon [\tau , \tau + \Delta] \rightarrow \Rn$ which is a solution of $\dot{x} = \D_t \, G(x, t)$ on $[\tau , \tau + \Delta]$ with $x(\tau) = x_\tau$.
\end{veta}

We focus on the standard theorems displayed above rather than more recent results concerned with existence and uniqueness of GODEs in Banach spaces by A.\ Slavík \cite{slavik6}. The reason for this is that while Theorem \ref{schwabunique} works identically in Banach spaces, the problem of existence in this setting has different specifics that are less suited as the starting point when building our theory.

\begin{priklad} \label{carath} Let $f \colon \Rn \times [a,b] \rightarrow \Rn$ satisfy $\mathbf{(C1)}$-$\mathbf{(C3)}$ and fix $t_0 \in [a,b]$. Then the equation $\dot{x} = f(x,t)$ in the sense of Carath\'{e}odory is equivalent to equation \eqref{geq} with
\begin{equation*}
G(x, t) = \mathrm{(L)} \int_{t_0}^t f(x,s) \ddd s.
\end{equation*}
Under these conditions, we have both \eqref{calf1} and \eqref{calf2}. If the function $f$ also satisfies  
\begin{equation} \label{genlip}
\| f(x,t) - f(y,t) \| \leq l(t) \, \omega(\| x - y \|)
\end{equation} 
for $x, y \in \Rn$ and almost all $t \in [a,b]$, where $l$ is Lebesgue integrable on $[a,b]$ and $\omega$ is a modulus function, then $G \in \mathcal{F}(\O, h, \omega)$ with
\begin{equation*}
h(t) = \int_{t_0}^{t} m(s) + l(s) \ddd s.
\end{equation*}
Particularly, if $f$ is generalized Lipschitz i.e.\ satisfies \eqref{genlip} with $\omega(t) = t$, then we also have \eqref{osgoodtype}. For additional details and proof, see Chapter 5 in \cite{schwabik92}.
\end{priklad}

We can easily observe that the behaviour of the function $t \mapsto F(x, \tau, t)$ directly translates into quality of the solution. From this perspective, condition \eqref{calf1} forces solutions to have bounded variation. This is not ideal, as we can see from the following example.

\begin{priklad} \label{bvpriklad}
Consider the function
\begin{equation*}
  y(s) = \left\{
  \begin{array}{l l}
    s^2 \, \cos(\pi / s^2) & \quad 0 < s \leq 1,\\
   0 & \quad s = 0. \\
  \end{array} \right.
\end{equation*} 
It has a derivative everywhere on $[0,1]$. Hence, it solves the equation \eqref{geq} with
\begin{equation*}
G(x, t) = \mathrm{(SHK)} \int^t_0 \dot{y}(s) \ddd s = y(t).
\end{equation*}
However, since it does not have bounded variation, $G$ does not satisfy \eqref{calf1}.
\end{priklad}
We would prefer to deal with solutions in the space of regulated functions i.e.\ both one-sided limits exist at every point in time. Regulated functions play a very important role in the theory of Kurzweil integration and generalized ordinary differential equations ( see e.g.\ \cite{tvrdybook} ). In fact, many key results suggest that solutions in the space of regulated functions are the expected ideal outcome.


\section{Solutions in metric spaces}

In the remainder of the text $X$ always denotes a metric space. We adopt the convention that $|x - y|$ stands for the distance between $x$ and $y$.

\begin{defi} \label{mainmetricsol}

Let $I \subset \R$ be an interval, $\Omega \subset X$ and assume that $F \colon \Omega \times I \times I \rightarrow X$ satisfies $F(x, \tau, \tau) = x$. We say that $u \colon I \rightarrow X$ is a solution of $ \dot{x} = \D_t \, F(x, \tau, t)$ on $I$ if $u(\tau) \in \Omega$ for all $\tau \in I$ and 
\begin{equation} \label{metricsol}
(\text{SHK})\int_{\alpha}^{\beta} \D_t \, |u(t) - F(u(\tau),\tau, t)| = 0 \quad \text{for} \quad [\alpha , \beta ] \subset I,
\end{equation}
i.e.\ constant zero is an indefinite SHK integral of $U(\tau, t) = |u(t) - F(u(\tau),\tau, t)|$ on~$I$.
\end{defi}

\begin{lemmz} \label{metricequiv1}
Let $\Omega \subset X$. Let $F \colon \Omega \times I \times I \rightarrow X$ satisfy $F(x, \tau, \tau) = x$. 
Let $[a,b] \subset I$ and $u \colon I \to \Omega$. Then the following statements are 
equivalent:
\begin{enumerate}[label=$\mathrm{(\roman*)}$, noitemsep]
\item 
For every $\varepsilon > 0$ there exists $\delta \colon [a,b] \rightarrow \R^+$ such that every $\delta$-fine partition $\{[t_{i-1}, t_i], \tau_i\}_{i = 1}^k$ of $[a,b]$ satisfies
\begin{equation} \label{skhsol-}
\sum_{i=1}^k  \Big|  | u(t_i) - F( u(\tau_i), \tau_i, t_i) | -  | u(t_{i-1})- F(u(\tau_i), \tau_i, t_{i-1}) | \Big| < \varepsilon.
\end{equation}
\item
For every $\varepsilon > 0$ there exists $\delta \colon [a,b] \rightarrow \R^+$ such that every $\delta$-fine partition $\{[t_{i-1}, t_i], \tau_i\}_{i = 1}^k$ of $[a,b]$ satisfies
\begin{equation} \label{skhsol+}
\sum_{i=1}^k  \Big(  | u(t_i) - F( u(\tau_i), \tau_i, t_i) | +  | u(t_{i-1})- F(u(\tau_i), \tau_i, t_{i-1}) | \Big) < \varepsilon.
\end{equation}
\end{enumerate}
\end{lemmz}

\begin{proof}
If $\{[t_{i-1}, t_i], \tau_i\}_{i=1}^k$ is $\delta$-fine, then the partition 
\begin{equation*}
\{[s_{j-1}, s_j], \sigma_j\}_{j=1}^{2k} = \{ \ldots, [t_{i-1}, \tau_i], \tau_i, [\tau_i, t_i], \tau_i, \ldots \}
\end{equation*}
is also $\delta$-fine. Hence, we get
\begin{equation} \label{skhsum}
\sum_{j=1}^{2k} \Big| |  u(s_j) - F( u(\sigma_j), \sigma_j, s_j) |  - | u(s_{j-1})- F(u(\sigma_j), \sigma_j, s_{j-1})  | \Big| < \varepsilon.
\end{equation}
For $j = 2i - 1$ we have
\begin{equation} \label{skhsub1}
\begin{split}
 | u(s_j) - F( u(\sigma_j), \sigma_j, s_j) |    & = |  u(\tau_i) - F( u(\tau_i), \tau_i, \tau_i) | = 0, \\
| u(s_{j-1})- F(u(\sigma_j), \sigma_j, s_{j-1})  | & = |  u(t_{i-1}) - F( u(\tau_i), \tau_i, t_{i - 1}) |.  
\end{split}
\end{equation}
Similarly, for $j = 2i$ we have
\begin{equation} \label{skhsub2}
\begin{split}
| u(s_j) - F( u(\sigma_j), \sigma_j, s_j) | & = |  u(t_i) - F( u(\tau_i), \tau_i, t_i) |, \\
| u(s_{j-1})- F(u(\sigma_j), \sigma_j, s_{j-1})  | & = |  u(\tau_i) - F( u(\tau_i), \tau_i, \tau_i) | = 0.
\end{split}
\end{equation}
By substituting \eqref{skhsub1} and \eqref{skhsub2} back into \eqref{skhsum} we get \eqref{skhsol+}. The reverse implication is trivial. 
\end{proof}

\begin{cor} Let $E$ be a normed linear space and $\Omega\subset E$.
Let $F \colon \Omega \times I \times I \rightarrow X$ satisfy $F(x, \tau, \tau) = x$
and $u\colon I\to \Omega$ be a function. Then $u$ is the solution of 
the equation $\dot x = D_tF(x,\tau, t)$ in the sense of Definition \ref{standardgode}
if and only if $u$ is its solution of the sense of Definition \ref{mainmetricsol}.
\end{cor}

\begin{proof}
If $\{[t_{i-1}, t_i], \tau_i\}_{i=1}^k$ is a partition of $[a,b] \subset I$, then for each $i = 1, \ldots , k$ we have
\begin{align*}
\Bigl| &\| u(t_i) - F( u(\tau_i), \tau_i, t_i) \| -  \| u(t_{i-1})- F(u(\tau_i), \tau_i, t_{i-1}) \| \Big| \\
&\quad \leq \big\| u(t_i) -u(t_{i-1}) - F( u(\tau_i), \tau_i, t_i) +   F(u(\tau_i), \tau_i, t_{i-1}) \big\| \\
&\quad \leq \| u(t_i) - F( u(\tau_i), \tau_i, t_i) \| +  \| u(t_{i-1})- F(u(\tau_i), \tau_i, t_{i-1}) \|.
\end{align*}
We add the above inequalities over $i = 1, \ldots , k$ and recall that Definition \ref{mainmetricsol} and Definition \ref{standardgode} correspond to the first and second sum respectively. We then see that one implication is a consequence of the first inequality while the reverse implication follows from the second inequality and Lemma \ref{metricequiv1}.
\end{proof}

\begin{pozn}
Using Theorem \ref{integralequiv} and normalization of $F$ we can immediately deduce that $u \colon I \rightarrow X$ is a solution of $ \dot{x} = \D_t \, F(x, \tau, t)$ on $I$ if and only if there exists an increasing function $\xi \colon I \rightarrow \R$ such that for all $\tau \in I$ we have $u(\tau) \in \Omega$ and 
\begin{equation} \label{mcsol}
\lim_{t \, \to \, \tau, \; t \, \in \, I} \frac{|u(t) - F(u(\tau),\tau,t)|}{\xi(t) - \xi(\tau)} = 0. 
\end{equation}
In this case we will refer to $\xi$ as control function of $(F, u)$ rather than $(U, 0)$ (where $U$ is once again the numerator of \eqref{mcsol}) and tacitly employ properties which the control function inherits from the MC integral.
\end{pozn}

\section{Uniqueness} \label{secunique}

\begin{defi} \label{unique-future}
Let $X$ be a metric space, $I \subset \R$ an interval and $u \colon I \rightarrow X$ a solution of $\dot{x} = \D_t \, F(x, \tau, t)$ on $[ \tau, \tau + \delta] \subset I$ with $u(\tau) = x_\tau$. We say that $u$ is locally unique in the future at the point $(x_\tau, \tau)$ if for any solution $v \colon I \rightarrow X$ of ${x} = \D_t \, F(x, \tau, t)$ on $[\tau, \tau + \eta]$ such that $v(\tau) = x_\tau$ there exists $\zeta > 0$ such that $v(t) = u(t)$ for $t \in [\tau, \tau + \zeta]$.
\end{defi}

\begin{lemmz} \label{thompson} Let $f \colon (a,b) \rightarrow \R$ be a real function which at every point $x \in (a,b)$ satisfies
\begin{enumerate}[label=$\mathbf{(T \arabic*)}$, noitemsep]
\item $\D^+ f(x) \geq 0$,
\item $\displaystyle \limsup_{h \to 0_+} f(x-h) \leq f(x)$.
\end{enumerate}

Then $f$ is nondecreasing on $(a,b)$.
\end{lemmz}
\begin{proof}
See \cite{thomson}, Theorem 55.10, page 135.
\end{proof}

\begin{veta} \label{mainunique}
Let $X$ be a metric space, $\Omega \subset X$, let $F \colon \Omega \times [a,b) \times [a,b) \rightarrow X$ fulfil $F(x, \tau, \tau) = x$ and let us assume the following: 
\begin{enumerate}[label=$\mathbf{(U \arabic*)}$, noitemsep]
\item For every $x, y \in \Omega$ and $\tau \in (a,b)$ let
\begin{equation*}
\liminf_{t \to \tau_-}  |F(x, \tau, t) - F(y, \tau, t)| - |x - y|  \geq 0.
\end{equation*}
\item There exists an increasing function $\xi \colon [a,b) \rightarrow \R$ and an Osgood type modulus function $\omega \colon \R^+_0 \rightarrow \R^+_0$ such that for every $x, y \in \Omega$ and $\tau \in (a,b)$ we have
\begin{equation*}
\liminf_{t \to \tau_+} \frac{ |F(x, \tau, t) - F(y, \tau, t)| - |x - y|  }{\xi(t) - \xi(\tau)} \leq \omega(|x-y|).
\end{equation*}
\end{enumerate}
Then every solution of $\dot{x} = \D_t \, F(x, \tau, t)$ on $[a,b)$ is locally unique in the future.
\end{veta}

\begin{proof}
Let $u, v$ be two solutions on $[a,b)$ and let $\alpha , \beta \in [a,b)$ be such that $\alpha < \beta$, $u(\alpha) = v (\alpha)$ and $u(\beta) \neq v(\beta)$. We also assume that $\xi$ already controls both $(F, u)$ and $(F, v)$ and that for $t > s$ we have
\begin{equation} \label{figl2}
\xi(t) - \xi(s) \geq t - s. 
\end{equation}
Condition $\mathbf{(U1)}$ implies that the function $\Delta(t) = |u(t) - v(t)|$ is lower semicontinuous from the left:  \par
Fix $\tau \in (a,b)$. Since both $u$ and $v$ solve $\dot{x} = \D_t \, F(x, \tau, t)$, we have
\begin{equation} \label{sol5}
\begin{split}
\lim_{t \to \tau} |u(t) - F(u(\tau), \tau, t)| = 0, \\
\lim_{t \to \tau} |v(t) - F(v(\tau), \tau, t)| = 0.
\end{split}
\end{equation}
Together with $\mathbf{(U1)}$ we obtain
\begin{align*}
&|u(\tau) - v(\tau)| \leq \liminf_{t \to \tau_-}  |F(u(\tau), \tau, t) - F(v(\tau), \tau, t)| \\
\leq &\liminf_{t \to \tau_-} \Big( |F(u(\tau), \tau, t) - u(t)| + | u(t) - v(t) | + | v(t) - F(v(\tau), \tau, t)| \Big)   \\
= &\liminf_{t \to \tau_-} |u(t) - v(t)|.
\end{align*}
Therefore, we can achieve $u \neq v$ everywhere on $(\alpha, \beta)$ by replacing $\alpha$ with the supremum of the set $\{t \in [\alpha, \beta) \cara u(t) = v(t)\}$. This is possible because $\Delta$ is lower semicontinuous from the left and $\Delta(\beta) > 0$. \par
We also notice that 
\begin{equation} \label{deltalimit}
\lim_{t \to \alpha_+} \Delta(t) = 0.
\end{equation}
This is easily observed from \eqref{sol5} and the fact that $u(\alpha) = v(\alpha)$, since
\begin{align*}
0 \leq |u(t) - v(t)| &\leq |u(t) - F(u(\alpha), \alpha, t)| + |F(u(\alpha), \alpha, t) - v(t)|   \\
&= |u(t) - F(u(\alpha), \alpha, t)| + |F(v(\alpha), \alpha, t) - v(t)|.
\end{align*}
Now, fix $\nu > 0$ and set 
\begin{equation*}
\Phi(r) = \int^{\nu}_{r} \frac{1}{\omega(s)} \,\mathrm{d}s.
\end{equation*}
Consider the function $\Psi(t) = \Phi(\Delta(t)) + 2 \, \xi(t)$ on $(\alpha, \beta)$. It is well defined everywhere on $(\alpha, \beta)$, because $\Delta(t) > 0$ for $t \in (\alpha, \beta)$. Since $\xi$ is bounded on $[\alpha,\beta]$, we use \eqref{osgoodtype}, \eqref{deltalimit} and the composite limit law to deduce
\begin{equation} \label{spor}
\lim_{t \rightarrow \alpha_+} \Psi(t) = + \infty.
\end{equation}
We now employ Lemma \ref{thompson} to show that $\Psi$ is nondecreasing on $(\alpha, \beta)$, which obviously contradicts \eqref{spor}. We begin by verifying $\mathbf{(T2)}$: \par
Fix $\tau \in (\alpha, \beta)$ and $\varepsilon > 0$. Find $\delta > 0$ such that for $|s - \Delta (\tau )| < \delta$ we have
\begin{equation} \label{sporhelp1}
| \Phi (s) - \Phi (\Delta (\tau ))| < \varepsilon.
\end{equation}
We have already established that $\Delta$ is lower semicontinuous from the left. Therefore, we can find $\gamma > 0$ such that all $t \in (\tau - \gamma , \tau )$ satisfy
\begin{equation} \label{sporhelp2}
\Delta (t) > \Delta (\tau) - \frac{\delta}{2}.
\end{equation}
Since $\Phi$ is decreasing, we use \eqref{sporhelp1} and \eqref{sporhelp2} to obtain
\begin{equation*}
\Phi (\Delta (t)) < \Phi (\Delta (\tau) - \frac{\delta}{2}) < \Phi (\Delta (\tau)) + \varepsilon.  
\end{equation*}
Combined with $\xi$ being increasing we have
\begin{align*}
\limsup_{t \to \tau_-} \Psi(t) &= \limsup_{t \to \tau_-} \Big( \Phi (\Delta (t)) + 2 \,  \xi(t) \Big) \\
&= \limsup_{t \to \tau_-} \Phi (\Delta (t)) + \lim_{t \to \tau_-} 2 \, \xi(t)  \\
&\leq \Phi (\Delta (\tau )) + 2 \, \xi(\tau) = \Psi(\tau).
\end{align*}
The rest of the proof will focus on showing that all $\tau \in (\alpha, \beta)$ satisfy 
\begin{equation} \label{upperder}
 \D^+ \, \Psi(\tau) \geq 0.
\end{equation}
We fix $\tau \in (\alpha, \beta)$ and distinguish two cases: \\[6pt]
$1)$ For every $\delta > 0$ there exists $t \in (\tau, \tau + \delta)$ such that $\Delta(t) \leq \Delta(\tau)$. This implies 
\begin{equation*}
\Phi(\Delta(t)) + 2 \, \xi(t) \geq \Phi(\Delta(\tau)) + 2 \, \xi(\tau),
\end{equation*}
which obviously results in \eqref{upperder}. \\[6pt]
$2)$ There exists $\delta_0 > 0$ such that for every $t \in (\tau, \tau + \delta_0)$ we have $\Delta(t) > \Delta(\tau)$.\\[6pt]
Since $\xi$ controls both $(F,u)$ and $(F, v)$, we can find $\delta_1 \in (0, \delta_0)$ such that all $t \in (\tau, \tau + \delta_1)$ satisfy
\begin{equation} \label{mcreseni}
\frac{|u(t) - F(u(\tau), \tau, t)|}{\xi(t) - \xi(\tau)} < \varepsilon, \quad \frac{|v(t) - F(v(\tau), \tau, t)|}{\xi(t) - \xi(\tau)} < \varepsilon.
\end{equation}
Now, fix $\varepsilon > 0$ such that
\begin{equation} \label{figl1}
3 \, \varepsilon < \omega(\Delta(\tau)). 
\end{equation} 
We use $\mathbf{(U2)}$ to find  $\delta_2 \in (0, \delta_1)$ such that for all $\delta \in (0, \delta_2)$ there exists $\tilde{\tau} \in (\tau, \tau + \delta)$ satisfying
\begin{equation} \label{podminka2}
\frac{ |F(u(\tau), \tau, \tilde{\tau}) - F(v(\tau), \tau, \tilde{\tau})| - \Delta(\tau)}{\xi(\tilde{\tau}) - \xi(\tau)} < \omega(\Delta(\tau)) + \varepsilon.
\end{equation}
Denote
\begin{align*}
E_u(t) &= |u(t) - F(u(\tau), \tau, t)|, \\
E_v(t) &= |v(t) - F(v(\tau),\tau, t)|  
\end{align*}
and notice that
\begin{equation} \label{doubletriangle}
\Delta(t) \leq |F(u(\tau), \tau, t) - F(v(\tau),\tau,t)| + E_u(t) + E_v(t).
\end{equation}
Applying \eqref{doubletriangle} to \eqref{podminka2} results in
\begin{equation} \label{odhad2}
\frac{ \Delta(\tilde{\tau}) - \Delta(\tau) - E_u(\tilde{\tau}) - E_v(\tilde{\tau})}{\xi(\tilde{\tau}) - \xi(\tau)} \leq \omega(\Delta(\tau)) + \varepsilon.
\end{equation}
We combine \eqref{odhad2} with \eqref{mcreseni} and obtain 
\begin{equation} \label{odhad}
\frac{ \Delta(\tilde{\tau}) - \Delta(\tau) }{\xi(\tilde{\tau}) - \xi(\tau)} \leq \omega(\Delta(\tau)) + 3 \, \varepsilon.
\end{equation}
Now, let $\sigma > \eta > 0$ and observe
\begin{equation*}
\Phi(\eta) - \Phi(\sigma) = \int^{\sigma}_{\eta} \frac{1}{\omega(s)} \,\mathrm{d}s \leq \int^{\sigma}_{\eta} \frac{1}{\omega(\eta)} \,\mathrm{d}s = \frac{1}{\omega(\eta)} (\sigma - \eta).
\end{equation*}
By choosing $\sigma = \Delta(\tilde{\tau})$ and $\eta = \Delta(\tau)$ we obtain 
\begin{equation} \label{lipsch}
\Phi(\Delta(\tau)) - \Phi(\Delta(\tilde{\tau})) \leq \frac{1}{\omega(\Delta(\tau))} (\Delta(\tilde{\tau}) - \Delta(\tau)).
\end{equation}
We combine \eqref{lipsch} with \eqref{odhad} and get
\begin{align*}
\omega(\Delta(\tau)) \Big( \Phi(\Delta(\tau)) - \Phi(\Delta(\tilde{\tau})) \Big) &\leq (\omega(\Delta(\tau)) + 3 \, \varepsilon)(\xi(\tilde{\tau}) - \xi(\tau)),  \\[4pt]
\Phi(\Delta(\tilde{\tau}))-\Phi(\Delta(\tau)) &\geq (- 1 - \frac{3 \, \varepsilon}{\omega(\Delta(\tau))})(\xi(\tilde{\tau}) - \xi(\tau)).
\end{align*}
Adding $2 \, \xi(\tilde{\tau}) - 2\, \xi(\tau)$ to the above inequality results in
\begin{align*}
\Psi(\tilde{\tau}) - \Psi(\tau) &\geq  ( 1 - \frac{3 \, \varepsilon}{\omega(\Delta(\tau))})(\xi(\tilde{\tau}) - \xi(\tau)), \\[4pt] 
\frac{\Psi(\tilde{\tau}) - \Psi(\tau)}{\xi(\tilde{\tau}) - \xi(\tau)} &\geq 1 - \frac{3 \, \varepsilon}{\omega(\Delta(\tau))}.
\end{align*}
Due to \eqref{figl1} and \eqref{figl2} we have
\begin{equation} \label{done}
\frac{\Psi(\tilde{\tau}) - \Psi(\tau)}{\tilde{\tau} - \tau} \geq  0.
\end{equation}
We have demonstrated that for every $\delta \in (0, \delta_2)$ there exists $\tilde{\tau} \in (\tau, \tau + \delta)$ satisfying \eqref{done}. Hence, we finally obtain \eqref{upperder} and the proof is finished.
\end{proof}

\begin{pozn}
It is interesting to note that Theorem \ref{mainunique} makes no demand on the quality of the space or the function $t \mapsto F(x, \tau, t)$ i.e.\ the quality of a single solution. Condition $\mathbf{(U1)}$ merely implies that the distance between two solutions is lower semicontinuous from the left. Investigating uniqueness without this property might be difficult, as we could then have two solutions that are equal on a subset of the real line but not in its supremum and we do not believe there is much room for improvement left in this regard. 
\end{pozn}


\section{Existence} \label{secexist}

Due to the mentioned difficulties with a stand-alone existence theorem in more abstract spaces, we will for now focus on Theorem \ref{schwabexist} i.e.\ existence in metric spaces that are locally compact. We begin by listing additional preliminaries. 

\subsection{Regulated functions}

\begin{defi} \label{equiregdef}
A function $f \colon [a,b] \rightarrow X$ is called regulated if the limit $f(t+)$ exists for all $t \in [a,b)$ and the limit $f(s-)$ exists for all $s \in (a,b]$. The set of all regulated functions from $[a,b]$ to $X$ will be denoted by $\mathcal{R}_X ([a,b])$. The set $\K \subset \mathcal{R}_X ([a,b])$ is called equiregulated if for every $\varepsilon > 0$ and $\tau \in [a,b]$ there exists $\delta > 0$ such that all $x \in \K$ satisfy
\begin{align*}
|x(t) - x(\tau -)| < \varepsilon \quad &\text{for} \quad t \in (\tau - \delta, \tau), \\
|x(t) - x(\tau +)| < \varepsilon \quad &\text{for} \quad t \in (\tau, \tau + \delta).
\end{align*}
\end{defi}

The following theorem concerns compact sets in the space of regulated functions. For $\Rn$-valued functions it is a special case of results presented by D.\ Fraňková in \cite{frankova}. Our case is significantly simplified by the assumption of compact target space. While the proof follows almost identical procedure, it will be displayed here for the convenience of the reader.

\begin{lemmz} \label{flemma1}
Assume that the set $\K \subset \mathcal{R}_X ([a,b])$ is equiregulated. Then for every $\varepsilon > 0$ there is a division $a = s_0 < s_1 < ... < s_k = b$ such that
\begin{equation} \label{franlemma1}
 |x(t) - x(s)| < \varepsilon \quad \text{for} \quad x \in \K, \, [t, s] \subset (s_{j-1}, s_j), \, j = 1,2,\ldots ,k.
\end{equation}
\end{lemmz}

\begin{proof}
For every $t \in [a,b]$ we can find $\delta(t) > 0$ corresponding to $\K$ and $\varepsilon$ by Definition \ref{equiregdef}. By Lemma \ref{cousin} there exists a $\delta$-fine partition $\{[t_{i-1}, t_i], \tau_i\}_{i=1}^m$ of $[a, b]$. Set
\begin{equation*}
\{s_j\}_{j=1}^{2m + 1} = \{t_0, \tau_1, t_1, \tau_2, \ldots , t_m\}.
\end{equation*}
If two or three consecutive points are identical, we remove all but the first. The resulting division obviously satisfies the desired condition.
\end{proof}

\begin{veta} \label{franthm}
Let $X$ be a compact metric space and $\K \subset \mathcal{R}_X ([a,b])$. If $\K$ is equiregulated, then it is relatively compact in the topology of uniform convergence.
\end{veta}

\begin{proof}
Showing that $\mathcal{R}_X ([a,b])$ is complete is a standard exercise. Our goal is to show that $\K$ is totally bounded.\par
Fix $\varepsilon > 0$. By Lemma \ref{flemma1} there exists a division $a = s_0 < s_1 < ... < s_k = b$ such that \eqref{franlemma1} holds with $\varepsilon / 2$ rather than $\varepsilon$. Let $\{\alpha_1, \ldots , \alpha_m \} \subset X$ be a finite $(\varepsilon/2)$-net of $X$ i.e.
\begin{equation*}
X \subset \bigcup_{i = 1}^m \B(\alpha_i, \frac{\varepsilon}{2}).  
\end{equation*}
Define $F$ as the set of all functions $x \colon [a,b] \rightarrow X$ such that $x$ is constant on $(s_{j - 1}, s_j)$ for every $j = 1, \ldots , k$ and $x(t) \in \{\alpha_1, \ldots , \alpha_m \}$ for every $t \in [a,b]$. This set is evidently finite and consists of regulated functions. We can easily verify that $F$ is an $\varepsilon$-net of $\K$.
\end{proof}

\begin{pozn} \label{equimodulus}
A common sufficient (in some cases also necessary) condition for a set of functions $\K$ to be equiregulated is the existence of an increasing function $h \colon [a,b] \rightarrow \R$ and a modulus function $\zeta$ such that for every $x \in \K$ and $t,s \in [a,b]$ we have
\begin{equation*}
|x(t) - x(s)| \leq \zeta ( |h(t) - h(s)| ).
\end{equation*}
Indeed, for every $\varepsilon > 0$ and $\tau \in [a,b)$ we can choose $\delta > 0$ such that
\begin{equation*}
|h(\tau+) - h(t)| < \zeta^{-1}( \varepsilon ) 
\end{equation*}
for $t \in (\tau, \tau + \delta)$, which obviously results in
\begin{equation*}
|x(\tau+) - x(t)| < \varepsilon.
\end{equation*}

\end{pozn}

\subsection{Main existence theorem}

\begin{veta} \label{mainexist}
Let $X$ be a locally compact metric space, let $\Omega \subset X$ be open, let $F \colon \Omega \times [a,b] \times [a,b] \rightarrow X$ satisfy $F(x, \tau, \tau) = x$, let $\tilde{x} \in \Omega$ satisfy $F(\tilde{x}, a, a+) \in \Omega$ and let us assume the following: 
\begin{enumerate}[label=$\mathbf{(E \arabic*)}$, noitemsep]
\item For every $\tau, t \in [a,b]$ the function $x \mapsto F(x, \tau, t)$ is continuous. 
\item There exists an increasing function $h \colon [a,b] \rightarrow \R$ and a modulus function $\zeta$ such that for every $x \in \Omega$ and every $\tau, t, s \in [a,b]$ we have
\begin{equation*}
|F(x, \tau, t) - F(x, \tau, s)| \leq \zeta ( |h(t) - h(s)| ).
\end{equation*} 
\item There exists a neighbourhood $\U \subset \Omega$ of $F(\tilde{x}, a, a+)$ such that for every $\tau \in (a,b]$ there exists $\mu_{\tau} > 0$ such that for every $x \in \U$ and $t \in (\tau - \mu_{\tau}, \tau)$ there exists $y \in \Omega$ with $F(y, \tau, t) = x$. 
\item There exists an increasing function $\xi \colon [a,b] \rightarrow \R$ and a modulus function $\omega \colon \R^+_0 \rightarrow \R^+_0$ such that for $\tau, \sigma, t, s \in [a,b]$ and $x, y \in \Omega$ we have
\begin{align*}
| F(x, \tau, t) - F(y, \sigma, t) | &\leq | F(x, \tau, s) - F(y, \sigma, s) | \\
&+ \omega (|x - y| + |\tau - \sigma|) \, |\xi(t) - \xi(s)|.
\end{align*}
\end{enumerate}
Then there exists $\Delta > 0$ and a solution $u \colon [a, a + \Delta] \rightarrow X$ of $\dot{x} = \D_t \, F(x, \tau, t)$ on $[a, a + \Delta]$ with $u(a) = \tilde{x}$.
\end{veta}

\begin{proof}

\textbf{Step 1}: Method of construction. \\[6pt]
First, we set
\begin{equation*}
x_1 = \lim_{t \to a_+} F(\tilde{x}, a, t).
\end{equation*}
Find $\lambda > 0$ such that $\B(x_1, \lambda)$ is compact and contained in $\U$. Set
\begin{equation} \label{definek}
K = \omega(3 \, \lambda + |\tilde{x} - x_1|). 
\end{equation}
Find $\Delta \in (0, \lambda)$ such that 
\begin{equation} \label{solcontain}
\zeta (|h(a+) - h(t)|) < \frac{\lambda}{10} \quad \mathrm{for} \quad t \in (a, a + \Delta].
\end{equation}
Furthermore, let
\begin{align} \label{vyberk}
|\xi(a+) - \xi(t)| < \frac{\lambda}{ 2 \, K } \quad \mathrm{for} \quad t \in (a, a + \Delta].
\end{align}
To simplify notation we will assume that $b = a + \Delta$. \par
Let $\delta \colon [a,b] \rightarrow \R^+$ be a positive function satisfying $\delta(t) \leq \mu_t$ for $t \in [a,b]$ ($\mu_0 = \Delta$). Find a $\delta$-fine partition $\A = \{[t_{j-1}, t_j], \tau_j\}_{j=1}^n$ of $[a,b]$ such that $a$ is a tag. We will construct an approximate solution $v \colon [a,b] \rightarrow X$ corresponding to $\A$. For $t \in [a, t_1]$ set $v(t) = F(\tilde{x}, a, t)$. Due to \eqref{solcontain} we have 
\begin{equation} \label{ocascontain}
|x_1 - F(\tilde{x}, a, t)| \leq \zeta( |h(a+) - h(t)| ) < \frac{\lambda}{ 10 } \quad \mathrm{for} \quad t \in (a, b].
\end{equation}
This implies that for $t_1$ in particular, we get $|x_1 - v(t_1)| \leq \lambda$. Therefore, condition $\mathbf{(E3)}$ implies that there exists $y \in \Omega$ such that $F(y, \tau_2 , t_1) = v(t_1)$. For $t \in [t_1, t_2]$ set $v(t) = F(y, \tau_2, t)$. Particularly $v(\tau_2) = y$. We can now use
\begin{align} \label{trefa}
F(\tilde{x}, a, t_1) = v(t_1) = F(v(\tau_2), \tau_2, t_1)
\end{align} 
in combination with \eqref{solcontain} to obtain
\begin{align*}
|\tilde{x} - v(\tau_2)| &\leq |\tilde{x} - x_1| + |x_1 - v(t_1)| + |v(t_1) - v(\tau_2)| \\
&= |\tilde{x} - x_1| + |F(\tilde{x}, a, a+) - F(\tilde{x}, a, t_1)| \\
&+ |F(v(\tau_2), \tau_2, t_1) - F(v(\tau_2), \tau_2, \tau_2)| \\
&\leq |\tilde{x} - x_1| + \zeta (|h(a+) - h(t_1)|) + \zeta (|h(t_1) - h(\tau_2)|) \\
&< |\tilde{x} - x_1| + \frac{2 \, \lambda}{10}.
\end{align*}
Note that parts of the previous estimate also imply $v(\tau_2) \in \U(x_1, \lambda) \subset \Omega$. Since $\omega$ is increasing, we have
\begin{align*}
\omega ( |\tilde{x} - v(\tau_2)| + |a - \tau_2| ) &< \omega(|\tilde{x} - v(\tau_2)| + \lambda) \\
&< \omega(|\tilde{x} - x_1| + \frac{2 \, \lambda}{10} + \lambda) \\
&< \omega(|\tilde{x} - x_1| + 3 \, \lambda) = K.
\end{align*}
Condition $\mathbf{(E4)}$ applied to $a, \tau_2, t,  t_1$ and $\tilde{x}, v(\tau_2)$ together with \eqref{trefa} implies that for $t \in [t_1, t_2]$ we have
\begin{equation} \label{ipcontain}
|F(\tilde{x}, a, t) - v(t)| \leq K \, |\xi(t_1) - \xi(t)|.
\end{equation}
Fix $j \leq n$ and assume that for $t \in [a, t_{j-1}]$ we have constructed $v(t)$ with \eqref{ipcontain}. We will find such $v(t)$ for $t \in [t_{j-1}, t_j]$. \par
By combining \eqref{vyberk}, \eqref{ocascontain} and \eqref{ipcontain} we get 
\begin{align} \label{lambdaodhad}
\begin{split}
|x_1 - v(t)| &\leq |x_1 - F(\tilde{x}, a, t)| + |F(\tilde{x}, a, t) - v(t)| \\
&\leq |x_1 - F(\tilde{x}, a, t)| + K \, |\xi(t_1) - \xi(t)|  \\
&< \frac{\lambda}{10} + \frac{\lambda}{2} < \lambda \quad \mathrm{for} \quad t \in (a, t_{j-1}].
\end{split}
\end{align}
Particularly $|x_1 - v(t_{j-1})| < \lambda$. Thus, condition $\mathbf{(E3)}$ allows us to find $y \in \Omega$ such that $F(y, \tau_j , t_{j-1}) = v(t_{j-1})$. Set $v(t) = F(y, \tau_j , t)$ for $t \in [t_{j-1}, t_j]$. Due to \eqref{lambdaodhad} and \eqref{solcontain} we have
\begin{align} \label{lambdawin}
\begin{split}
|x_1 - v(t)| &\leq |x_1 - v(t_{j-1})| + |v(t_{j-1}) - v(t)| \\
&< \frac{\lambda}{10} + \frac{\lambda}{2} + \zeta(|h(t_{j-1}) - h(t)|) \\
&< \lambda \quad \mathrm{for} \quad t \in [t_{j-1}, t_j].
\end{split}
\end{align}
Since $\tau_j \in [t_{j-1}, t_j]$, we obtain
\begin{equation} \label{lambdaodhad2}
\begin{split}
|\tilde{x} - v(\tau_j)| &\leq |\tilde{x} - x_1| + |x_1 - v(\tau_j)| \\
&< |\tilde{x} - x_1| + \lambda.
\end{split}
\end{equation}
For $t \in [t_{j-1}, t_j]$ we can apply condition $\mathbf{(E4)}$ to $a, \tau_j, t, t_1$ and $\tilde{x}, v(\tau_j)$. Combined with \eqref{lambdaodhad2}, \eqref{ipcontain} and \eqref{definek} we get
\begin{align*}
|F(\tilde{x}, a, t) - v(t)| &\leq |F(\tilde{x}, a, t_{j-1}) - v(t_{j-1})| \\
&+ \omega(|\tilde{x} - v(\tau_j)| + |a - \tau_j|) \, |\xi(t_{j-1}) - \xi(t)| \\
&\leq K \, |\xi(t_1) - \xi(t_{j-1})| + K \, |\xi(t_{j-1}) - \xi(t)| \\
&= K \, |\xi(t_1) - \xi(t)| \quad \mathrm{for} \quad t \in [t_{j-1}, t_j].
\end{align*}
This finishes the construction. From \eqref{lambdawin} we learn that the resulting approximate solution $v \colon [a,b] \rightarrow X$ satisfies
\begin{equation} \label{ballcontain}
|x_1 - v(t)| < \lambda \quad \mathrm{for} \quad t \in (a,b].
\end{equation}
\textbf{Step 2}: Approximate solutions. \\[6pt]
Set $\varepsilon_k = 2^{-k}$ and find $\eta_k > 0$ such that for $|x - y| < \eta_k$ and $|\tau - \sigma | < \eta_k$ we have
\begin{equation} \label{etakchoice}
\omega (|x - y| + |\sigma - \tau|) < \varepsilon_k.
\end{equation}
For $w \in [a,b]$ set 
\begin{equation*}
J(\tau) = \max \, \{ \zeta ( |h(\tau) - h(\tau -)| ), \zeta ( |h(\tau) - h(\tau +)| ) \}.
\end{equation*}
Set
\begin{equation*} 
\psi_k = \min \, \Big\{ \frac{\varepsilon_k}{6}, \frac{\eta_k}{6} \Big\}.
\end{equation*}
For every $\tau \in [a,b]$ use $\mathbf{(E2)}$ to find $\delta_k(\tau) > 0$ such that
\begin{align}  \label{e4prep1}
\begin{split}
\zeta ( |h(t) - h(\tau -)| ) &< \frac{\psi_k}{2} \quad \mathrm{for} \quad t \in (\tau - \delta_k(\tau), \tau), \\
\zeta ( |h(s) - h(\tau +)| ) &< \frac{\psi_k}{2} \quad \mathrm{for} \quad s \in (\tau, \tau + \delta_k(\tau)).
\end{split}
\end{align}
We also arrange that for every $\tau \in [a,b]$ we have
\begin{align}
\delta_k (\tau) &< \frac{\eta_k}{2}, \label{deltaeta} \\
\delta_k (\tau) &\leq \delta_{k-1} (\tau), \label{deltadelta} \\
\delta_k (\tau) &\leq \mu_{\tau}. \notag
\end{align}
For each $k \in \N$ find a $\delta_k$-fine partition $\A_k$ of $[a,b]$ (with $a$ as a tag) and use it to construct an approximate solution $u_k \colon [a, b] \rightarrow X$ according to Step 1. \par
From \eqref{deltadelta} we learn that for $k, k_0 \in \N$ with $k \geq k_0$ the partition $\A_k$ is $\delta_{k_0}$-fine. This implies that if $J(\tau) \geq \psi_{k_0}$, then $\tau$ is a tag of $\A_k$ for all $k \geq k_0$, since assuming otherwise would contradict \eqref{e4prep1}. We can even say that if $\tau \in (a,b)$ and $J(\tau) \geq \psi_{k_0}$ then there exists $\mu > 0$ such that
\begin{equation*}
u_k(t) = F(u_k(\tau), \tau, t) \quad \mathrm{for} \quad t \in [\tau - \mu, \tau + \mu]
\end{equation*}
i.e.\ there either exists an index $i$ such that $\tau = \tau_i$ and $t_{i - 1} < \tau_i < t_i$ or there exists an index $j$ such that $\tau = \tau_j = t_j = \tau_{j+1}$. \par 
Finally, note that while we could outright demand $\delta_k \rightarrow 0$ pointwise, it is already contained as a consequence of \eqref{deltaeta} and \eqref{etakchoice}.
\\[6pt]
\textbf{Step 3}: Find a solution candidate. \\[6pt]
We will prove that $\{u_k\}_{k \in \N}$ is equiregulated. Fix $\varepsilon > 0$ and $\varkappa \in [a,b]$. Find $k_0 \in \N$ such that $\varepsilon_{k_0} < \varepsilon$. Set $K = \omega(3 \, \lambda + |\tilde{x} - x_1|)$ and find $\delta_\varkappa > 0$ such that 
\begin{align} \label{deltakappa}
\begin{split}
K \, |\xi(t) - \xi(\varkappa -)| &< \frac{\varepsilon}{6} \quad \mathrm{for} \quad t \in (\varkappa - \delta_\varkappa, \varkappa), \\
K \, |\xi(s) - \xi(\varkappa +)| &< \frac{\varepsilon}{6} \quad \mathrm{for} \quad s \in (\varkappa, \varkappa + \delta_\varkappa).
\end{split}
\end{align}
Since $u_k$ are regulated for $k \in \N$ due to being segments of $t \mapsto F(x, \tau, t)$ and Remark \ref{equimodulus}, we can find the corresponding $\nu_k > 0$ such that 
\begin{align*}
|u_k(t) - u_k(\varkappa -)| &< \varepsilon \quad \mathrm{for} \quad t \in (\varkappa - \nu_k, \varkappa),\\
|u_k(s) - u_k(\varkappa +)| &< \varepsilon \quad \mathrm{for} \quad s \in (\varkappa, \varkappa + \nu_k).
\end{align*}
Set $\delta = \min \, \{\nu_1, \ldots , \nu_{k_0}, \delta_{k_0}(\varkappa), \delta_\varkappa \}$. Fix $k \in \N$ with $k \geq k_0$. We distinguish two cases:\\[6pt]
$\mathbf{I)}$ Let $J(\varkappa) \geq \psi_{k_0}$. This means that $\varkappa$ is a tag of $\A_k = \{[t_{i-1}, t_i], \tau_i\}_{i = 1}^m$. Particularly, if $\varkappa \neq a$ then there exists $j \leq m$ such that $\varkappa = \tau_j$ and $t_{j-1} < \tau_j$.  \par
We recall that $0 < \delta < \delta_{k_0}(\varkappa)$ and that for $t \in (\varkappa - \delta_{k_0}(\varkappa), \varkappa)$ we have from \eqref{e4prep1} the inequality 
\begin{align} \label{regocasek}
\begin{split}
| F(u_k(\varkappa), \varkappa, t) - F(u_k(\varkappa), \varkappa, \varkappa-) | &\leq \zeta ( |h(t) - h(\varkappa -)| ) \\
&< \psi_{k_0} \leq \frac{\varepsilon_{k_0}}{6} < \frac{\varepsilon}{6}.
\end{split}
\end{align}
We further divide the first case into two possibilities: \\[6pt]
$\mathbf{i)}$ Let $|t_{j-1} - \varkappa| \geq \delta$. For $t \in [t_{j-1}, \varkappa]$ we have $u_k(t) = F(u_k(\varkappa), \varkappa, t)$. Since $\A_k$ is $\delta_{k_0}$-fine, we know that $|t_{j-1} - \varkappa| < \delta_{k_0}$. Now, we are done due to \eqref{regocasek}, because $(\varkappa - \delta, \varkappa) \subset [t_{j-1}, \varkappa)$ and for $t \in [t_{j-1}, \varkappa)$ we have  
\begin{equation*}
|u_k(t) - u_k(\varkappa-)| = | F(u_k(\varkappa), \varkappa, t) - F(u_k(\varkappa), \varkappa, \varkappa-) | < \varepsilon.
\end{equation*}
$\mathbf{ii)}$ Let $|t_{j-1} - \varkappa| < \delta$. We recall that $b - a \leq \lambda$ and that by \eqref{ballcontain} we have
\begin{equation*} 
|u_k(\tau) - u_k(\sigma)| < |\tilde{x} - x_1| + \lambda \quad \mathrm{for} \quad \tau, \sigma \in [a,b].
\end{equation*}
Consequently, we get
\begin{equation*} 
\omega(|u_k(\tau) - u_k(\sigma)| + |\tau - \sigma|) < K.
\end{equation*}
Applying $\mathbf{(E4)}$ to $\tau, \sigma, t, s \in [a,b]$ and $u_k(\tau), u_k(\sigma)$ results in
\begin{align} \label{e4use3}
\begin{split}
&| F(u_k(\tau), \tau, t) - F(u_k(\sigma), \sigma, t) | \\
\leq \; &| F(u_k(\tau), \tau, s) - F(u_k(\sigma), \sigma, s) | \\
+ \; &\omega (|u_k(\tau) - u_k(\sigma)| + |\tau - \sigma|) \, |\xi(t) - \xi(s)| \\
\leq \; &| F(u_k(\tau), \tau, s) - F(u_k(\sigma), \sigma, s) | + K \, |\xi(t) - \xi(s)|.
\end{split}
\end{align}
Particularly, for $i \in \N$ with $i < j$ and $t \in [t_{i-1}, t_i]$ we have
\begin{align*}
|u_k(t) - F(u_k(\varkappa), \varkappa, t)| &= | F(u_k(\tau_i), \tau_i, t) - F(u_k(\tau_j), \tau_j, t) | \\
&\leq | F(u_k(\tau_i), \tau_i, t_i) - F(u_k(\tau_j), \tau_j, t_i) | \\
&+ K \, |\xi(t) - \xi(t_i)|  \\
&= |u_k(t_i) - F(u_k(\varkappa), \varkappa, t_i)| + K \, |\xi(t) - \xi(t_i)|.
\end{align*}
We use $u_k(t_{j-1}) = F(u_k(\varkappa), \varkappa, t_{j-1})$ and simple induction to obtain
\begin{gather*}
|u_k(t) - F(u_k(\varkappa), \varkappa, t)| \leq K \, |\xi(t) - \xi(t_{j-1})| \quad \mathrm{for} \quad t \in [a, t_{j-1}].
\end{gather*}
For $t \in (\varkappa - \delta, \varkappa)$ we combine this with \eqref{regocasek} and \eqref{deltakappa} to get
\begin{align*}
|u_k(t) - u_k(\varkappa-)| &\leq |u_k(t) - F(u_k(\varkappa), \varkappa, t)| \\
&+ | F(u_k(\varkappa), \varkappa, t) - F(u_k(\varkappa), \varkappa, \varkappa-) | \\
&< K \, |\xi(t) - \xi(t_{j-1})| + \frac{\varepsilon}{6}  \\
&< \frac{\varepsilon}{6} + \frac{\varepsilon}{6} < \varepsilon.
\end{align*}
$\mathbf{II)}$ Let $J(\varkappa) < \psi_{k_0}$. We can find $\vartheta_k \in (\varkappa - \delta, \varkappa)$ such that
\begin{equation*}
|u_k(t) - u_k(\varkappa-)| < \frac{\varepsilon}{6} \quad \mathrm{for} \quad t \in [\vartheta_k, \varkappa).
\end{equation*}
Combined with $J(\varkappa) < \psi_{k_0}$ and \eqref{regocasek} we have
\begin{align} \label{zachytka}
\begin{split}
|u_k(t) - F(u_k(\varkappa), \varkappa, t)| &\leq |u_k(t) - u_k(\varkappa-)| + |u_k(\varkappa-) - u_k(\varkappa)|\\
&+ |u_k(\varkappa) - F(u_k(\varkappa), \varkappa, \varkappa-)| \\
&+ |F(u_k(\varkappa), \varkappa, \varkappa-) - F(u_k(\varkappa), \varkappa, t)| \\
&< \frac{\varepsilon}{6} + \psi_{k_0} + \psi_{k_0} + \frac{\varepsilon}{6} \\
&< \frac{4 \, \varepsilon}{6} \quad \mathrm{for} \quad t \in [\vartheta_k, \varkappa).
\end{split}
\end{align}
We can again inductively apply \eqref{e4use3} to obtain
\begin{align} \label{inductivebound}
\begin{split}
|u_k(t) - F(u_k(\varkappa), \varkappa, t)| &\leq |u_k(\vartheta_k) - F(u_k(\varkappa), \varkappa, \vartheta_k)| \\
&+ K \, |\xi(t) - \xi(\vartheta_k)| \quad \mathrm{for} \quad t \in [a, \vartheta_k].
\end{split}
\end{align}
We now apply \eqref{inductivebound}, \eqref{zachytka} with $t = \vartheta_k$, \eqref{deltakappa} and \eqref{regocasek} to obtain
\begin{align*}
|u_k(t) - u_k(\varkappa-)| &\leq |u_k(t) - F(u_k(\varkappa), \varkappa, t)| \\
&+ |F(u_k(\varkappa), \varkappa, t) - u_k(\varkappa-)| \\
&< \frac{4 \, \varepsilon}{6} + \frac{\varepsilon}{6} + \frac{\varepsilon}{6} = \varepsilon \quad \mathrm{for} \quad t \in (\varkappa - \delta, \vartheta_k].
\end{align*}
An analogous procedure can be applied on the right side of every $\varkappa \in [a,b)$ to show that $\{u_k\}_{k \in \N}$ is indeed equiregulated. \par
Therefore, the set $\{u_k\}_{k \in \N}$ satisfies the assumptions of Theorem \ref{franthm} on the compact space $\{\tilde{x}\} \cup \B(x_1, \lambda)$. Consequently, it is relatively compact. To simplify notation, we rename the convergent subsequence to $u_k$. The limit function will be denoted as $u$. \\[6pt]
\textbf{Step 4}: Prove that $u$ solves $\dot{x} = \D_t \, F(x, \tau, t)$.\\[6pt]
Our goal is to show that for every $\varepsilon > 0$ there exists $\delta \colon [a,b] \rightarrow (0, \infty)$ such that for any $\delta$-fine partition $\A = \{[t_{i-1}, t_i], \tau_i\}_{i = 1}^m$ of $[a,b]$ we have
\begin{equation} \label{reseni}
\sum_{i=1}^m  \Big( |u(t_{i-1}) - F(u(\tau_i), \tau_i, t_{i-1} )| + |u(t_i) - F(u(\tau_i), \tau_i, t_i )| \Big) < \varepsilon.
\end{equation}
Denote the left side of \eqref{reseni} by
\begin{equation*}
\sum_\A (u, F).
\end{equation*}
First, recall that for $t \in [a,b]$ we have
\begin{equation} \label{limres}
u(t) = \lim_{k \to \infty} u_k(t).
\end{equation}
For a fixed partition $\A$ of $[a,b]$ we can combine \eqref{limres} and $\mathbf{(E1)}$ to get
\begin{align} \label{limitrightside}
\begin{split}
 F(u(\tau_i), \tau_i, t_{i-1} ) &=  \lim_{k \to \infty} F(u_k(\tau_i), \tau_i, t_{i-1} ), \\
 F(u(\tau_i), \tau_i, t_i ) &= \lim_{k \to \infty} F(u_k(\tau_i), \tau_i, t_i ).
\end{split}
\end{align}
By \eqref{limres}, \eqref{limitrightside} and the continuity of the metric we have
\begin{align*}
|u(t_{i-1}) - F(u(\tau_i), \tau_i, t_{i-1} )| &= \lim_{k \to \infty} |u_k(t_{i-1}) - F(u_k(\tau_i), \tau_i, t_{i-1} )|,   \\
|u(t_i) - F(u(\tau_i), \tau_i, t_i )| &= \lim_{k \to \infty} |u_k(t_i) - F(u_k(\tau_i), \tau_i, t_i )|.  
\end{align*}
Using the additivity of limits we obtain
\begin{equation} \label{sumconv}
\sum_\A (u, F) = \lim_{k \to \infty} \sum_\A (u_k, F).
\end{equation}
Assume that we have the following property: 
\begin{itemize}[label=$(\star)$, noitemsep]
\item For every $\varepsilon > 0$ there exists $\delta \colon [a,b] \rightarrow \R^+ $ such that for every $\delta$-fine partition $\A$ of $[a,b]$ we we can find $k_5 \in \N$ with 
\begin{equation} \label{partialsolv}
 \sum_\A (u_k, F) < \varepsilon \quad \mathrm{for} \quad k \in \N, k \geq k_5.
\end{equation}
\end{itemize}
Then for any $\delta$-fine partition $\A$ we could use \eqref{sumconv} to find $k_6 \in \N, k_6 \geq k_5$ such that
\begin{equation*}
\sum_\A (u, F) < \sum_\A (u_{k_6}, F) + \varepsilon.
\end{equation*}
Together with \eqref{partialsolv} we would obtain the desired result
\begin{equation*}
\sum_\A (u, F) < \sum_\A (u_{k_6}, F) + \varepsilon < 2\, \varepsilon.
\end{equation*}
We now finish the proof by verifying $(\star)$. Let $\varepsilon > 0$ be given. Find $k_0 \in \N$ such that
\begin{equation*}
 2^{-k_0} < \min \Big\{ \varepsilon, \frac{\varepsilon}{2 \,(\xi(b) - \xi(a))} \Big\} . 
\end{equation*}
Denote by $E(r)$ the number of $\tau \in [a,b]$ such that $J(\tau) > r$. Find $\delta(\tau) < \delta_{k_0} (\tau)$ such that for all $t \in (\tau - \delta(\tau), \tau)$, $s \in (\tau, \tau + \delta(\tau))$ and $k \in \N$ we have
\begin{align} \label{e4prep2}
\begin{split}
K \, |\xi(t) - \xi(\tau -)| &< \frac{\varepsilon} {4 E(\psi_{k_0} )},  \\
K \, |\xi(s) - \xi(\tau +)| &< \frac{\varepsilon} {4 E(\psi_{k_0} )},   \\
|u_k(t) - u_k(\tau -)| &< \frac{\eta_{k_0}}{6}, \\
|u_k(s) - u_k(\tau +)| &< \frac{\eta_{k_0}}{6}. 
\end{split}
\end{align}
Let $\A = \{[t_{i-1}, t_i], \tau_i\}_{i = 1}^m$ be an arbitrary $\delta$-fine partition of $[a,b]$. By $\delta < \delta_{k_0}$ it is also $\delta_{k_0}$-fine. Therefore, we can once again use \eqref{e4prep1} to say that for all $\tau \in [a, b]$ with $J(\tau) \geq \psi_{k_0}$ we have $\tau \in \A$ i.e.\ there exists $i \leq m$ for which $\tau = \tau_i$. For all such $\tau \in (a,b)$ we eliminate the possibility $t_i = \tau_i = \tau = \tau_{i+1}$, since assuming $t_{i-1} < \tau_i = \tau < t_i$ does not change the final sum. We further recall that $\delta_k \rightarrow 0$ pointwise. Therefore, we can find $k_1 \geq k_0$ such that all $\tau_i \in \A$ with $J(\tau_i) \geq \psi_{k_0}$ satisfy
\begin{equation}  \label{badjumpcluster}
\delta_{k_1} (\tau_i) < \min \{|t_{i-1} - \tau_i|, |\tau_i - t_i| \}.
\end{equation}
The special cases (if necessary) obviously reduce to
\begin{equation*}
\delta_{k_1} (a) < |a - t_1| \quad \text{and} \quad \delta_{k_1} (b) < |t_{m-1} - b|.
\end{equation*}
Fix $k \geq k_1$, let $\A_k = \{ [s_{j-1}, s_j], \sigma_j\}_{j=1}^n$. Once again $J(\tau) \geq \psi_{k_0}$ implies $\tau \in \A_k$ i.e.\ every bad jump is a tag of both $\A$ and $\A_k$. We want to show that 
\begin{equation*}
\sum_{i=1}^m \Big( |u_k(t_{i-1}) - F(u_k(\tau_i), \tau_i, t_{i-1} )| + |u_k(t_i)  - F(u_k(\tau_i), \tau_i, t_i )| \Big) < \varepsilon.
\end{equation*}
Similarly to the previous step, we consider small and large jumps separately: \\[6pt]
$\mathbf{A)}$ Let $J(\tau_v) \geq \psi_{k_0}$ and let $w \leq n$ be such that $\tau_v = \sigma_w$. For $\tau_v \neq a$ we can write
\begin{equation*}
[t_{v - 1}, \tau_v] = \bigcup_{j = 1}^{m} \; [t_{v - 1}, \tau_v ] \cap [s_{j-1}, s_j].
\end{equation*}
Set $j_v = \min \{ j \leq n \cara (t_{v-1}, \tau_v) \cap (s_{j-1}, s_j) \neq \varnothing \}$. From \eqref{badjumpcluster} we get $j_v < w$. Thus, we can write
\begin{align*}
&|u_k(t_{v-1}) - F(u_k(\tau_v), \tau_v, t_{v-1} )| \\
= \; &|u_k(t_{v-1}) - F(u_k(\tau_v), \tau_v, t_{v-1} )| - |u_k(s_{j_v}) - F(u_k(\tau_v), \tau_v, s_{j_v} )| \\
+ \; &|u_k(s_{j_v}) - F(u_k(\tau_v), \tau_v, s_{j_v} )| - |u_k(s_{j_v + 1}) - F(u_k(\tau_v), \tau_v, s_{j_v + 1} )| \\
&\phantom{|u_k(s_{j_S (i)}) - F(u_k(\tau_i), \tau_i, s_{j_S (i)} )|} \vdots \\
+ \; &|u_k(s_{w - 2}) - F(u_k(\tau_v), \tau_v, s_{w - 2} )| - |u_k(s_{w - 1}) - F(u_k(\tau_v), \tau_v, s_{w - 1} )| \\
+ \; &|u_k(s_{w - 1}) - F(u_k(\tau_v), \tau_v, s_{w - 1} )|.
\end{align*}
Since $u_k(s_{w - 1}) = F(u_k(\tau_v), \tau_v, s_{w - 1} )$ we have
\begin{equation*}
 |u_k(s_{w - 1}) - F(u_k(\tau_v), \tau_v, s_{w - 1} )| = 0.
\end{equation*}
Once again, we inductively apply \eqref{e4use3} to obtain
\begin{align*}
|u_k(t_{v-1}) - F(u_k(\tau_v), \tau_v, t_{v-1} )| &\leq K \, |h(t_{v - 1}) - h(s_{j_v})| \\
&+ K \, |h(s_{j_v}) - h(s_{j_v + 1})| \\
&\quad \quad \quad \quad \quad \quad \vdots \\
&+ K \, |h(s_{w - 2}) - h(s_{w - 1})| \\
&= K \, |h(t_{v - 1}) - h(s_{w - 1})| \\
&< K \, |h(t_{v - 1}) - h(\tau_v -)|  \\
&< \frac{\varepsilon} {4 E(\psi_{k_0})}.
\end{align*}
Using the identical estimate on the right side, we conclude that the total damage of bad jumps to the final sum is less than $\varepsilon / 2$. \\[6pt]
$\mathbf{B)}$ Now let $\tau_v \in \A$ be such that $J(\tau_v) < \psi_{k_0}$. Fix $\sigma_z \in \A_k$ such that $(t_{v - 1}, t_v) \cap (s_{z-1}, s_z) \neq \varnothing$. We observe that $J(\sigma_z) < \psi_{k_0}$. If this was not the case, there would exist $\tau_q \in \A$ such that $\tau_q = \sigma_z$ and by \eqref{badjumpcluster} we would have $[s_{z - 1}, s_z] \subset (t_{q-1}, t_q)$, causing a contradiction. Additionally, we assume that $\sigma_z > \tau_v$ to avoid technical complications, since the opposite case is a direct analogy.  \par
Choose $\varkappa \in (\tau_v, t_v) \cap (s_{z-1}, \sigma_z)$ if both intervals are nonempty, $\varkappa = \tau_v$ if $\tau_v = t_v$ and $\varkappa = \sigma_z$ if $s_{z-1} = \sigma_z$. If both equalities hold we obtain a contradiction. Due to \eqref{deltaeta} we have
\begin{align*}
|\tau_v - \sigma_z| &\leq |\tau_v - \varkappa| + |\varkappa - \sigma_z |  \\
&< \delta(\tau_v) + \delta_k (\sigma_z) < 2 \, \delta_{k_0} < \eta_{k_0}.
\end{align*}
By \eqref{e4prep1} and \eqref{e4prep2} we get
\begin{align*}
|u_k(\tau_v) - u_k(\sigma_z)| &\leq |u_k(\tau_v) - u_k(\varkappa)| + |u_k(\varkappa) - u_k(\sigma_z) |  \\
&\leq |u_k(\tau_v) - u_k(\tau_v +)| + |u_k(\tau_v +) - u_k(\varkappa)| \\
&+ |u_k(\varkappa) - u_k(\sigma_z -)| + |u_k(\sigma_z -) - u_k(\sigma_z)| \\
&< \psi_{k_0} + \frac{\eta_{k_0}}{6} + \psi_{k_0} + \psi_{k_0} < \eta_{k_0}.
\end{align*}
Here, note that the difference between \eqref{e4prep1} and \eqref{e4prep2} is that $u_k$ consists of only one segment of $F$ on $[ s_{z-1}, s_z ]$ but possibly of multiple segments on $[t_{v - 1}, t_v]$. \par
For $t, s \in [t_{v - 1}, t_v] \cap [s_{z - 1}, s_z]$ we can now apply $\mathbf{(E4)}$ to $\tau_v$, $\sigma_z$, $t$, $s$ and $u_k(\tau_v)$, $u_k(\sigma_z)$ to obtain
\begin{align} \label{finale4}
\begin{split}
| F(u_k(\tau_v), \tau_v, t) - u_k(t) | &\leq | F(u_k(\tau_v), \tau_v, s) - u_k(s) | \\
&+ \varepsilon_{k_0} \, |\xi(t) - \xi(s)|.
\end{split}
\end{align}
We can once again write
\begin{equation*}
[t_{v - 1}, t_v] = \bigcup_{j = 1}^{m} \; [t_{v - 1}, t_v ] \cap [s_{j-1}, s_j]
\end{equation*}
and use the additivity of the right side of \eqref{finale4} to extend it to any $t, s \in [t_{v - 1}, t_v]$. Particularly for $t = t_{v - 1}, s = \tau_v$ and $t = t_v, s = \tau_v$ we utilize the normalization of $F$ to get
\begin{align*}
| F(u_k(\tau_v), \tau_v, t_{v - 1}) - u_k(t_{v - 1}) | &\leq \varepsilon_{k_0} \, |\xi(t_{v - 1}) - \xi(\tau_v)|, \\
| F(u_k(\tau_v), \tau_v, t_v) - u_k(t_v) | &\leq \varepsilon_{k_0} \, |\xi(\tau_v) - \xi(t_v)|.
\end{align*}
Consequently, the sum over all good tags is less than $\varepsilon_{k_0} \, ( \xi(b) - \xi(a) ) < \varepsilon/2$ and the proof is finished. 

\end{proof}

\begin{pozn}
The only purpose of condition $\mathbf{(E3)}$ is to construct solutions from Kurzweil type partitions and we believe it could eventually be eliminated by using a more sophisticated construction. 
\end{pozn}

\section{Linear case} \label{seclinear}

In this section we compare theorems presented in Sections \ref{secunique} and \ref{secexist} with standard results shown in Section \ref{secgode}. 

\subsection{Uniqueness} Here, we show that Theorem \ref{schwabunique} is contained in Theorem \ref{mainunique}, i.e.\ if the assumptions of Theorem \ref{schwabunique} are satisfied, then the assumptions of Theorem \ref{mainunique} are satisfied as well. We first take a look at how the assumptions of Theorem \ref{mainunique} can be simplified in normed linear spaces. \par
Let $E$ be a normed linear space, $\Omega \subset E$ and $F \colon \Omega \times (a,b) \times (a,b) \rightarrow E$. Set $\tilde{F}(x, \tau, t) = x + F(x, \tau, t) - F(x, \tau, \tau)$ and observe
\begin{align*}
&\hspace*{80pt} \| \tilde{F}(x, \tau, t) - \tilde{F}(y, \tau, t) \| \\
= \; &\| x + F(x, \tau, t) - F(x, \tau, \tau) - y - F(y, \tau, t) + F(y, \tau, \tau) \| \\
\geq \; &\|x - y\| - \| F(x, \tau, t) - F(x, \tau, \tau) - F(y, \tau, t) + F(y, \tau, \tau) \|. 
\end{align*}
Similarly
\begin{align*}
&\hspace*{80pt} \| \tilde{F}(x, \tau, t) - \tilde{F}(y, \tau, t) \| \\
= \; &\| x + F(x, \tau, t) - F(x, \tau, \tau) - y - F(y, \tau, t) + F(y, \tau, \tau) \| \\
\leq \; &\|x - y\| + \| F(x, \tau, t) - F(x, \tau, \tau) - F(y, \tau, t) + F(y, \tau, \tau) \|.
\end{align*}
Thus, our theorem takes the following form.
\begin{veta} \label{mainunique-lin}
Let $E$ be a normed linear space, $\Omega \subset E$ and $F \colon \Omega \times [a,b) \times [a,b) \rightarrow E$. Let us further assume the following: 
\begin{enumerate}[label=$\mathbf{( \widehat{U \arabic*})}$, noitemsep]
\item For every $x, y \in \Omega$ and $\tau \in (a,b)$ let
\begin{equation*}
\lim_{t \to \tau_-}  \| F(x, \tau, t) - F(x, \tau, \tau) - F(y, \tau, t) + F(y, \tau, \tau) \|  = 0.
\end{equation*}
\item There exists an increasing function $\xi \colon [a,b) \rightarrow \R$ and an Osgood type modulus function $\omega$ such that for every $x, y \in \Omega$ and $\tau \in (a,b)$ we have
\begin{equation*}
\liminf_{t \to \tau_+} \frac{ \| F(x, \tau, t) - F(x, \tau, \tau) - F(y, \tau, t) + F(y, \tau, \tau) \|  }{\xi(t) - \xi(\tau)} \leq \omega(\| x - y \|).
\end{equation*}
\end{enumerate}
Then every solution of $ \dot{x} = \text{D}_t F(x, \tau, t)$ on $[a,b)$ is locally unique in the future.
\end{veta}
Now, let $F(x, \tau, t) = G(x, t)$ satisfy the assumptions of Theorem \ref{schwabunique} i.e.\ let it belong to $\mathcal{F}(\O, h, \omega)$ where $h$ is continuous from the left and $\omega$ is an Osgood type modulus function. Condition \eqref{calf2} directly implies $\mathbf{(\widehat{U2})}$ with the same $\omega$ and $\xi = h$. In order to verify $\mathbf{(\widehat{U1})}$, we notice that continuity of $h$ from the left implies
\begin{equation*}
\lim_{t \to \tau_-} \omega (\|x - y\|) \, |h(t) - h(\tau)| = 0.
\end{equation*}
Therefore, condition \eqref{calf2} gives us
\begin{equation*}
\lim_{t \to \tau_-} \| G(x, t) - G(x, \tau) - G(y, t) + G(y, \tau) \| = 0.
\end{equation*}
Now, let us address condition \eqref{initialjump}. Theorem \ref{schwabunique} does not assume that $x$ is a solution on the right side of $\tau$. Since Theorem \ref{schwabexist} has weaker assumptions than Theorem \ref{schwabunique}, condition \eqref{initialjump} is present to ensure continuation of the solution. However, if the solution does exist on a right neighbourhood and we have other means to keep it contained within the domain, we need not assume the existence of the limit in \eqref{initialjump}. In such case, we can completely abstain from using condition \eqref{calf1}. \par
It is also worth addressing the difference between conditions \eqref{calf2} and $\mathbf{(U2)}$. The transition to a local version is merely a specific of nonrestricted equations. However, the transition from supremum to infimum is a substantial difference.

\subsection{Existence}
Here, we show that Theorem \ref{schwabexist} is contained in Theorem \ref{mainexist}. Let $\Omega \subset \Rn$ be open and let $F \colon \Omega \times [a,b] \times [a,b] \rightarrow \Rn$ satisfy
\begin{gather} 
\| F(x, \tau, t) - F(x, \tau, s) \| < \zeta ( |h(t) - h(s)| ), \label{calf1weak} \tag{$\mathcal{F}^*_1$}\\
\| F(x, \tau, t) - F(y, \sigma, t) - F(x, \tau, s) + F(y, \sigma, s) \| \label{calf2weak}  \tag{$\mathcal{F}^*_2$}\\
\leq \omega ( \| x - y \| + |\tau - \sigma |) \, |\xi(t) - \xi(s)| \notag
\end{gather}
for all $x, y \in \Omega$ and $\tau, \sigma, t, s \in [a,b]$ where $h$ and $\xi$ are increasing real functions on $[a,b]$ and $\omega$ as well as $\zeta$ are modulus functions. Set $\tilde{F}(x, \tau, t) =  x + F(x, \tau, t) - F(x, \tau, \tau)$. \\[6pt]
\textbf{e1)} We begin by showing that $x \mapsto \tilde{F}(x, \tau, t)$ is continuous for every $\tau, t \in [a,b]$. For $t = \tau$ it is the identity mapping. For $t \neq \tau$ we have
\begin{align*}
&\|x + F(x, \tau, t) - F(x, \tau, \tau) - y - F(y, \tau, t) + F(y, \tau, \tau) \|  \\
\leq \; &\| x - y \| + \| F(x, \tau, t) - F(x, \tau, \tau) - F(y, \tau, t) + F(y, \tau, \tau) \| \\
\leq \; &\|x - y\| + \omega(\|x - y\| + |\tau - \tau|) \, |\xi(t) - \xi(\tau)|. 
\end{align*}
We finish by observing that
\begin{equation*}
\lim_{ \| x - y \| \, \to \, 0 } \|x - y\| + \omega(\|x - y\|) \, |\xi(t) - \xi(\tau)| = 0.
\end{equation*}
\textbf{e2)} Choose $x \in \U_R$ and $\tau, t, s \in [a,b]$. We have
\begin{align*}
\noalign{\centering $\|\tilde{F}(x, \tau, t) - \tilde{F}(x, \tau, s)\|$ } 
= \; &\| x + F(x, \tau, t) - F(x, \tau, \tau) - x - F(x, \tau, s) + F(x, \tau, \tau) \| \\
= \; &\| F(x, \tau, t) - F(x, \tau, s) \| \leq \zeta ( |h(t) - h(s)| ).
\end{align*}
\textbf{e3)} Fix $\hat{x} \in \U_R$ and $\hat{\tau} \in [a,b)$. There exists $\alpha > 0$ such that $\B(\hat{x}, \alpha) \subset \U_R$. We know that there exists $\beta > 0$ such that deg$(f, \U (\hat{x}, \alpha), x) = 1$ whenever $x \in \U(\hat{x}, \beta)$ and $f \colon \B(\hat{x}, \alpha) \rightarrow \Rn $ is continuous with $\|f(y) - y\| < \beta$. Due to \textbf{(e2)} we know that
\begin{equation*}
\|\tilde{F}(x, \tau, t) - \tilde{F}(x, \tau, \tau +)\| \leq \zeta ( |h(t) - h(\tau +)| ).
\end{equation*}
Find $\gamma > 0$ such that $\zeta ( |h(t) - h(\hat{\tau} +)| ) < \beta $ for $t \in (\hat{\tau}, \hat{\tau} + \gamma )$. Fix $\tau, t \in (\hat{\tau}, \hat{\tau} + \gamma )$. We recall that $\tilde{F}(x, \tau, \tau) = x$ and that $h$ is increasing and obtain
\begin{equation*}
 \|\tilde{F}(x, \tau, t) - x\| \leq \zeta ( |h(t) - h(\tau )| ) < \beta. 
\end{equation*}
In \textbf{(e1)} we have shown that $x \mapsto \tilde{F}(x, \tau, t)$ is continuous. We also arrange $\gamma < \beta$ and obtain the following property: For every $\hat{x} \in  \U_R$ and $\hat{\tau} \in [a,b)$ there exists $\gamma > 0$ such that for any $x \in \U(\hat{x}, \gamma)$ and $\tau, t \in (\hat{\tau}, \hat{\tau} + \gamma)$ there exists $y \in \U_R$ such that $\tilde{F}(y, \tau, t) = x$. \par
By using this property for $\hat{\tau} = a$ and arranging $b < a + \gamma$ we obtain $\mathbf{(E3)}$ with $\U = \U( \tilde{F}(\tilde{x},a, a+), \gamma)$ and $\mu_{\tau} = \tau - a$.\\[6pt]
\textbf{e4)} Finally, we have
\begin{align*}
\noalign{\centering $\|\tilde{F}(x, \tau , t) - \tilde{F}(y, \sigma , t)\| - \|\tilde{F}(x, \tau , s) - \tilde{F}(y, \sigma , s)\| $}
= \; &\|x + F(x, \tau, t) - F(x, \tau, \tau ) - y - F(y, \sigma, t) + F(y, \sigma, \sigma ) \|  \\
- \; &\|x + F(x, \tau, s) - F(x, \tau, \tau) - y - F(y, \sigma, s) + F(y, \sigma, \sigma ) \|  \\
\leq \; &\| F(x, \tau, t) - F(y, \sigma, t) - F(x, \tau, s) + F(y, \sigma, s) \| \\
\leq \; &\omega(\|x - y\| + |\tau - \sigma| ) \, |\xi(t) - \xi(s)|.
\end{align*}
Now, we can see that our existence theorem takes the following form in $\Rn$.
\begin{veta} \label{mainexist-lin}
Let $\Omega \subset \Rn$ be open, let $F \colon \Omega \times [a,b] \times [a,b] \rightarrow \Rn$ satisfy \eqref{calf1weak} and \eqref{calf2weak} and let $\tilde{x} \in \Omega$ satisfy
\begin{equation*}
\lim_{t \to a_+} \tilde{x} + F(\tilde{x}, a, t) - F(\tilde{x}, a, a) \in \Omega.
\end{equation*}
Then there exists $\Delta > 0$ and a solution $u \colon [a, a + \Delta] \rightarrow X$ of $\dot{x} = \D_t \, F(x, \tau, t)$ on $[a, a + \Delta]$ with $u(a) = \tilde{x}$.
\end{veta}
Since in the restricted case \eqref{calf1} and \eqref{calf2} directly imply \eqref{calf1weak} and \eqref{calf2weak}, we come to the conclusion that Theorem \ref{schwabexist} is contained in Theorem \ref{mainexist}. To demonstrate the difference, we will consider the following existence theorem from \cite{henstocklectures} by R.~Henstock. We will show that while it is not included in Theorem \ref{schwabexist}, it does follow from Theorem \ref{mainexist}.

\begin{veta} \label{henstockexist}
Assume that $f \colon \Rn \times [a,b] \rightarrow \Rn$ satisfies the following conditions:
\begin{enumerate}[label=$\mathbf{(H \arabic*)}$, noitemsep] 
\item The function $x \mapsto f(x, t)$ is continuous for almost all $t \in [a,b]$.
\item The function $t \mapsto f(x, t)$ is SHK integrable over $[a,b]$ for every $x \in \Rn$.
\item There exists $S \subset \Rn$ compact and $\delta \colon [a,b] \rightarrow \R^+$ such that all $\delta$-fine partitions $\{\alpha = \alpha_0$, $\tau_1$, $\alpha_1, \ldots , \tau_k$, $\alpha_k = \beta \}$ of $[\alpha,\beta] \subset [a,b]$ and all functions $w \colon [a,b] \rightarrow \Rn$ satisfy
\begin{equation*} 
\sum_{i=1}^k f(w(\tau_i), \tau_i) (\alpha_i - \alpha_{i-1}) \in S. 
\end{equation*}
\end{enumerate}
Then for every $v \in \Rn$ and $\tau \in [a,b]$ there exists $y \colon [a,b] \rightarrow \Rn$ such that 
\begin{equation*}
y(t) = v + \mathrm{(SHK)} \int_{\tau}^t f(y(s), s) \ddd s \quad \text{for} \quad t \in [a,b].
\end{equation*} 
\end{veta}
We prove that $\mathbf{(H1)}$-$\mathbf{(H3)}$ for $f \colon \Rn \times [a,b] \rightarrow \Rn$ imply $\mathbf{(E1)}$-$\mathbf{(E4)}$ for 
\begin{equation*}
F(x, \tau, t) = x + \mathrm{(SHK)} \int^t_{\tau} f(x, s) \ddd s.
\end{equation*}
Since condition $\mathbf{(H3)}$ ensures that the solution stays in a compact set around the initial condition $(v, \tau)$, we can limit ourselves to studying the function $f$ on $\B_R :=\{ x \in \Rn \cara \| x \| \leq R \} $ for $R > 0$ sufficiently large. We make use of the following decomposition theorem (\cite{schwabik92}, page 78).

\begin{veta} \label{decompose}
A function $f \colon \Rn \times [a,b] \rightarrow \Rn$ satisfies $\mathbf{(H1)}$-$\mathbf{(H3)}$ if and only if $f(x, t) = g(t) + h(x,t)$, where $g \colon [a,b] \rightarrow \Rn$ is SHK integrable over $[a,b]$ and $h \colon \Rn \times [a,b] \rightarrow \Rn$ satisfies the Carath\'{e}odory conditions $\mathbf{(C1)}$-$\mathbf{(C3)}$.
\end{veta}
We already mentioned that $\mathbf{(C1)}$-$\mathbf{(C3)}$ for $h$ imply \eqref{calf1} and \eqref{calf2} for
\begin{equation*}
G(x, t) = \mathrm{(L)} \int_{t_0}^t h(x, s) \ddd s.
\end{equation*}
We notice that
\begin{equation*}
F(x, \tau, t) = x + G(x, t) - G(x, \tau) + \mathrm{(SHK)} \int_{\tau}^t g(s) \ddd s.
\end{equation*}
Hence
\begin{align*}
 F(x, \tau, t_2) - F(x, \tau, t_1) &= G(x, t_2) - G(x, t_1) + \mathrm{(SHK)} \int_{t_1}^{t_2} g(s) \ddd s, \\
 F(y, \sigma, t_1) - F(y, \sigma, t_2) &=  G(y, t_1) - G(y, t_2) + \mathrm{(SHK)} \int_{t_2}^{t_1} g(s) \ddd s.
\end{align*}
By adding these equalities we infer that \eqref{calf2weak} is indifferent to $g$. Since $G$ satisfies \eqref{calf1} and the SHK integral of $g$ is continuous, we have that $F$ satisfies \eqref{calf1weak}. Consequently, Theorem \ref{henstockexist} is contained in Theorem \ref{mainexist}. To see that it is not contained in Theorem \ref{schwabexist}, we can consider Example \ref{bvpriklad}. \par
While Theorem \ref{henstockexist} is able to handle functions of unbounded variation, it remains within the confines of the standard ordinary differential equation theory. It allows for the solution to be an indefinite SHK integral of the right hand side, but not in the sense of coupled variables. However, the above method for dealing with an error function which does not depend on the space variable gives clear indication on how to modify any GODE example from $\mathcal{F} (\O, h, \omega)$ so that it no longer satisfies \eqref{calf1}, while \eqref{calf1weak} still holds. Thus, we can see that, even in the context of Euclidean spaces, Theorem \ref{mainexist} contains both Theorem \ref{schwabexist} and Theorem \ref{henstockexist}, but it is not covered by them. Moreover, it invites the question of whether we could deal with regulated solutions in spaces that are not locally compact, which will be pursued in future research. \\[10pt]
\textbf{Acknowledgement.} I would like to thank Jan Malý for inspiration and guidance during the creation of this text.


\end{document}